\documentclass[12pt]{amsart}

\usepackage{amsmath,amsthm,amssymb,mathrsfs,graphicx,tikz,setspace,verbatim}
\usepackage[margin=1in]{geometry}
\usepackage[hidelinks]{hyperref}

\newtheorem{theorem}{Theorem}[section]
\newtheorem{lemma}[theorem]{Lemma}
\newtheorem{proposition}[theorem]{Proposition}
\newtheorem{corollary}[theorem]{Corollary}

\theoremstyle{definition} 
\newtheorem{definition}[theorem]{Definition}
\newtheorem{example}[theorem]{Example}
\newtheorem{remark}[theorem]{Remark}
\newtheorem{conjecture}[theorem]{Conjecture}
\newtheorem{proposal}[theorem]{Proposal}

\DeclareMathOperator{\Alg}{Alg}
\DeclareMathOperator{\Cob}{2Cob}

\DeclareMathOperator{\ext}{ext}
\DeclareMathOperator{\id}{id}
\DeclareMathOperator{\im}{im}

\DeclareMathOperator{\intt}{int}
\DeclareMathOperator{\open}{open}

\DeclareMathOperator{\Rep}{Rep}
\DeclareMathOperator{\SAlg}{SAlg}

\DeclareMathOperator{\spann}{span}

\newcommand{\C}{\mathbb{C}}
\newcommand{\Cc}{\mathcal{C}}

\newcommand{\cpp}{\mathrm{closed}++}
\newcommand{\Dc}{\mathcal{D}}
\newcommand{\F}{\mathbb{F}}
\newcommand{\g}{\mathfrak{g}}
\newcommand{\gl}{\mathfrak{gl}}

\newcommand{\opp}{\mathrm{open}++}
\newcommand{\Pc}{\mathcal{P}}

\newcommand{\pgl}{\mathfrak{pgl}}
\newcommand{\psl}{\mathfrak{psl}}
\newcommand{\Q}{\mathbb{Q}}
\newcommand{\R}{\mathbb{R}}
\newcommand{\Sc}{\mathcal{S}}

\newcommand{\Z}{\mathbb{Z}}
\newcommand{\Zb}{\mathbf{Z}}

\newcommand{\ootimes}{ 
  \mathbin{
    \mathchoice
      {\buildcircleotimes{\displaystyle}}
      {\buildcircleotimes{\textstyle}}
      {\buildcircleotimes{\scriptstyle}}
      {\buildcircleotimes{\scriptscriptstyle}}
  } 
}

\newcommand\buildcircleotimes[1]{%
  \begin{tikzpicture}[baseline=(X.base), inner sep=0, outer sep=0]
    \node[draw,circle] (X)  {$#1\otimes$};
  \end{tikzpicture}%
}

\title[Actions of both E and F]{Decategorified Heegaard Floer theory and actions of both E and F}
\author[Andrew Manion]{Andrew Manion}
\address{Department of Mathematics, North Carolina State University, 2108 SAS Hall, Raleigh, NC 27695}
\email{ajmanion@ncsu.edu}

\begin{document}

\begin{abstract}
    We define larger variants of the vector spaces one obtains by decategorifying bordered (sutured) Heegaard Floer invariants of surfaces. We also define bimodule structures on these larger spaces that are similar to, but more elaborate than, the bimodule structures that arise from decategorifying the higher actions in bordered Heegaard Floer theory introduced by Rouquier and the author. In particular, these new bimodule structures involve actions of both odd generators $E$ and $F$ of $\mathfrak{gl}(1|1)$, whereas the previous ones only involved actions of $E$. Over $\F_2$, we show that the new bimodules satisfy the necessary gluing properties to give a 1+1 open-closed TQFT valued in graded algebras and bimodules up to isomorphism; in particular, unlike in previous related work we have a gluing theorem when gluing surfaces along circles as well as intervals. Over the integers, we show that a similar construction gives two partially-defined open-closed TQFTs with two different domains of definition depending on how parities are chosen for the bimodules. We formulate conjectures relating these open-closed TQFTs with the $\mathfrak{psl}(1|1)$ Chern--Simons TQFT recently studied by Mikhaylov and Geer--Young.
\end{abstract}

\maketitle

\tableofcontents

\section{Introduction}

This paper aims to address one of the most natural and commonly-asked questions about work of Rouquier and the author \cite{ManionRouquier} as well as related work \cite{LaudaManion,ManionTrivalent, ChangManion,ManionDHA,ManionDHASigns}: out of the generators $E,F,K_1^{\pm 1}, K_2^{\pm 1}$ of $U_q(\gl(1|1))$, \cite{ManionRouquier} and these other papers feature higher actions of $E$ and not $F$ (or $F$ and not $E$ in different conventions) on the strands algebras of bordered sutured Heegaard Floer homology. Duals of $E$ feature prominently in the story but upon decategorification they satisfy a different type of relation with $E$ than do $E,F \in U_q(\gl(1|1))$. Why don't we have both $E$ and $F$? This paper proposes an explanation as well as, at the decategorified level, a modified setup in which actions of both $E$ and $F$ appear naturally. In this modified setup, we will prove gluing theorems when gluing surfaces along circles as well as intervals, further generalizing the types of surface gluing considered in \cite{ManionRouquier,ManionDHA,ManionDHASigns}.

\smallskip

\noindent \textbf{Intervals and the positive half.} Two important features of the higher actions in \cite{ManionRouquier} are as follows:
\begin{itemize}
    \item They are actions of a categorification of
    \[
    U(\psl(1|1)^+) = \C[E]/(E^2)
    \]
    where $\psl(1|1)$ is the Lie superalgebra generated by two odd elements $E$ and $F$ with vanishing superbracket and $\psl(1|1)^+$ is its positive half, the Lie superalgebra generated by one odd element $E$ with vanishing superbracket.

\smallskip
    
    \item They are associated to intervals, not circles, in the boundaries of surfaces with corners.
\end{itemize} 
We would like to suggest that these two features are closely related in the context of TQFTs extended down to a point, with the inverval viewed as the identity cobordism from a point to itself. 

Loosely speaking, when a fully extended 3d TQFT assigns the representation category (with tensor structure) of some Hopf algebra $H$ to a point, it will then assign the representation category of the Drinfeld double $D(H)$ to a circle. The 3d Chern--Simons or Reshetikhin--Turaev TQFT for a Lie algebra $\g$ assigns a semisimplified category of representations of $U_q(\g)$ to a circle, and $U_q(\g)$ is ``almost'' the Drinfeld double of $U_q(\g^+)$ where $\g^+$ is the positive half of $\g$. However, in general $U_q(\g)$ is not actually a Drinfeld double, the modular category $\Cc$ assigned to the circle by Reshetikhin--Turaev is not actually a Drinfeld center, and Reshetikhin--Turaev theory cannot actually be described as a fully extended TQFT assigning some representation category of $U_q(\g^+)$ to a point. By contrast, the Turaev--Viro theory assigning $\Cc$ to a point should be fully extended and assign the Drinfeld center of $\Cc$ (in some sense $\Cc^2$) to a circle, and in this sense the Reshetikhin--Turaev theory for $\Cc$ is a square root of the Turaev--Viro theory for $\Cc$ at the level of 1+1+1 extended TQFTs.

Now, for the special case $\g = \psl(1|1)$, we have $U(\psl(1|1)) \cong D(U(\psl(1|1)^+))$, so (ignoring $q$) the above difficulty in extending Reshetikhin--Turaev theory to a point using the positive half $\psl(1|1)^+$ should disappear. We make the following imprecise conjecture.

\begin{conjecture}\label{conj:GeneralPtConjImprecise}
    If we take $H = U(\psl(1|1)^+) = \C[E]/(E^2)$ with $\Delta(E) = E \otimes 1 + 1 \otimes E$ and $\varepsilon(E) = 0$, then the TQFT assigning $(\Rep(H),\otimes)$ to a point recovers both:
    \begin{itemize}
        \item aspects of the $\psl(1|1)$ Chern--Simons TQFT as studied by Mikhaylov \cite{Mikhaylov} and Geer--Young \cite{GeerYoung};
        \item aspects of decategorified bordered sutured Heegaard Floer theory in dimensions $1$ and $2$ as studied in \cite{ManionDHA,ManionDHASigns}
    \end{itemize}
    which are thereby closely related to each other.
\end{conjecture}

Note that the TQFT assigning $(\Rep(H), \otimes)$ to a point will assign $\Rep(H)$ as a bimodule category over itself to an interval, and it will assign $\Rep(D(H))$ to a circle. For appropriate surfaces $F$ with corners, then, we should expect actions of $U(\psl(1|1)^+)$ for intervals in $\partial F$ and actions of $U(\psl(1|1))$ for circles in $\partial F$. The first type of action is what gets categorified by the higher actions of \cite{ManionRouquier}. 

\begin{proposal}\label{proposal:EBecausePoint}
    There is ``$E$ but not $F$'' in \cite{ManionRouquier} because the higher actions and tensor product of \cite{ManionRouquier}, as well as the earlier work of \cite{DM}, are more related to the Heegaard Floer homology of the point and the interval than to the Heegaard Floer homology of the circle.
\end{proposal}

\begin{remark}
In the Acknowledgments section of \cite{DM}, Douglas--Manolescu say that their work in cornered Heegaard Floer homology, upon which \cite{ManionRouquier} builds, was inspired by a question that David Nadler asked of Manolescu: ``What is the Seiberg--Witten invariant of the circle?'' In our proposed interpretation, Douglas--Manolescu's work instead concerns the Seiberg--Witten or Heegaard Floer invariant of a point, which would be something like 
\[
(2\Rep(\mathcal{U}(\psl(1|1)^+)), \ootimes)
\]
where $\mathcal{U}(\psl(1|1)^+)$ is the dg monoidal category called $\mathcal{U}$ in \cite{ManionRouquier} and $\ootimes$ is the higher tensor product operation defined in \cite{ManionRouquier}. If our interpretation is close to accurate, it would give a very nice conceptual interpretation for the higher tensor product operation of \cite{ManionRouquier} and its relationship to Heegaard Floer homology: very roughly, we propose that
\begin{center}
\emph{The higher tensor product $\ootimes$ of \cite{ManionRouquier} is the key ingredient in the Heegaard Floer homology of a point.}
\end{center}
\end{remark}

\begin{remark}
There is another perspective from which this proposal is not entirely implausible. Heegaard Floer homology itself is defined based on the ansatz that Seiberg--Witten theory, extended down to surfaces, assigns to a surface $\Sigma$ the Fukaya category of a symmetric power of $\Sigma$. Correspondingly, when defining Heegaard Floer homology for a 3-manifold, one picks a decomposition along a 2d Heegaard surface and works on this surface (which is one lower dimension than expected for 3-manifold invariants, and this drop in dimension is because one is utilizing extended TQFT structure).

Now, both Lipshitz--Ozsv{\'a}th--Thurston's and Zarev's variants of bordered Heegaard Floer theory work by choosing an \emph{additional cut on the 3-manifold}, transverse to the Heegaard surface and intersecting the Heegaard surface in a 1-manifold. To define bordered Floer invariants of surfaces, one works with this 1-manifold (in the guise of a pointed matched circle or an arc diagram), and to define bordered Floer invariants of 3d cobordisms, one works with a Heegaard surface having 1d boundary given by the pointed matched circle or arc diagram. It is tempting to think that bordered Heegaard Floer theory is again utilizing extended TQFT structure of one lower dimension (one higher level of extension) than expected, and indeed, by \cite{ManionRouquier} the bordered Floer surface invariants are objects of a 2-representation 2-category that should be assigned to the 1-manifold underlying a pointed matched circle or arc diagram.

Finally, the tensor product of \cite{ManionRouquier} and the corresponding gluing formula for bordered Floer surface invariants was prefigured by Douglas--Manolescu's theory of cornered Heegaard Floer homology \cite{DM}, which yet again is based on making another cut on a 3-manifold (transverse to both the Heegaard surface and the above ``bordered'' cut). While Douglas--Manolescu work with Lipshitz--Ozsv{\'a}th--Thurston's variant of bordered Floer theory and have a different topological interpretation of what's going on, \cite{ManionRouquier} works with Zarev's variant and gets a gluing formula for algebras associated to 1d arc diagrams when gluing like ``interval = interval $\cup_{\mathrm{pt}}$ interval'' (for the surfaces represented by the arc diagrams, this gives the open pair-of-pants gluing involving the $p=2$ case of Example~\ref{ex:PantsIntro} below). Specifically, the algebra for the glued arc diagram is the higher tensor product of the algebras for the two pieces, which categorifies what one would expect from the TQFT assigning $(\Rep(U(\mathfrak{psl}(1|1)^+)),\otimes)$ to a point.
\end{remark}

\smallskip

\noindent \textbf{Main results.} We now state the main results of this paper, which will be phrased in terms of open-closed TQFT rather than TQFT extended down to a point. We plan to return to TQFT extended down to a point in \cite{ManionPoint} (in preparation).

In \cite{ManionRouquier, ManionDHA,ManionDHASigns} we do not have the expected actions of $U(\psl(1|1))$ for circles; we only have actions of $U(\psl(1|1)^+)$ for intervals. Indeed, while one would expect the identity cobordism $\id_{S^1}$ on $S^1$ to get assigned the identity functor on $\Rep(U(\psl(1|1)))$, which is tensor product with $U(\psl(1|1))$ as a bimodule over itself (4-dimensional), the vector space associated to this cobordism in \cite{ManionDHA,ManionDHASigns} only has dimension $2$. In this paper we will define related spaces that have the expected actions of both $U(\psl(1|1)^+)$ for intervals and $U(\psl(1|1))$ for circles.

Let $\Cob^{\ext}$ denote the $1+1$ open-closed cobordism category defined in \cite{LaudaPfeiffer}. For an object $M$ of $\Cob^{\ext}$ consisting of a disjoint union of oriented intervals and circles in some specified order, define $A_{EF}(M)$ to be the tensor product (in order) of super rings 
\[
U^{\Z}(\psl(1|1)^+) := \Z[E]/(E^2)
\]
for interval components of $M$ and 
\[
U^{\Z}(\psl(1|1)) := \Z\langle E,F \rangle / (E^2, F^2, EF + FE)
\]
for circle components of $M$,  where $\Z \langle \cdots \rangle$ denotes noncommutative polynomials. Give these super rings $\Z$-gradings by setting $\deg(E) = -1$ and $\deg(F) = 1$. 

We will first state our main result over $\F_2$; let $A^{\F_2}_{EF}(M) = A_{EF}(M) \otimes \F_2$ viewed as an ordinary non-super $\Z$-graded algebra. Let $F \colon M_1 \to M_2$ be a morphism in $\Cob^{\ext}$ and let $S_+ =M_2 \sqcup (-M_1)$. Let $S_- $ denote the closure of $\partial F \setminus S_+$. Let $P$ be a collection of points, one in each $S_+$ boundary component of $F$ (the below constructions will be independent of $P$ up to isomorphism). For any rational number\footnote{Rather than restricting to $\Q$-gradings, we could more generally choose $A \in \R$ or $A \in \C$ if we wanted. On the other hand, the choices of $A$ we will be most concerned with are $A = 1/2$ and $A=1$. For these choices the $\Q$-grading is really just a $\frac{1}{4}\Z$-grading or a $\frac{1}{2}\Z$-grading respectively.} $A \in \Q$, let
\begin{equation}\label{eq:DeltaADef}
\begin{aligned}
    \delta_{A}(F) = &-A h + (A - 1)\#\{\textrm{no-}S_+\textrm{ non-closed components of }F\} \\ 
    &+ A\#\{\textrm{no-}S_-\textrm{ non-closed components of }F\} \\
    &+ (2A-1)\#\{\textrm{closed components of }F\}\\
    &+ ((A-1)/2) \#\{S_+\textrm{ intervals}\} - (1/2)\#\{S_+\textrm{ circles}\}.
\end{aligned}
\end{equation} We will define the structure of a bimodule over $(A^{\F_2}_{EF}(M_2), A^{\F_2}_{EF}(M_1))$ on the $\Q$-graded vector space
\[
\Zb^P_{\delta_A,\F_2}(F) := \wedge^* H_1(F,P; \F_2) \{\delta_A(F)\},
\]
where $\{\cdot\}$ denotes a shift in the $\Q$-grading and the summand $\wedge^k$ of the exterior algebra lives in $\Q$-degree $k$. Note that $\wedge^* H_1(F,P)$ is higher-dimensional than the space $\wedge^* H_1(F,S_+)$ featuring in \cite{ManionDHA,ManionDHASigns} if and only if $F$ has at least one component $F_0$ with an $S_+$ boundary circle that is not the only boundary circle of $F_0$.

\begin{theorem}\label{thm:IntroMain1F2}
Let $M_1$, $M_2$, and $M_3$ be objects of $\Cob^{\ext}$ and let
\[
M_3 \xleftarrow{F'} M_2 \xleftarrow{F} M_1
\]
be morphisms in $\Cob^{\ext}$. For any $A \in \Q$, we have
\[
\Zb^P_{\delta_A,\F_2}(F' \circ F) \cong \Zb^P_{\delta_A,\F_2}(F') \otimes_{A^{\F_2}_{EF}(M_2)} \Zb^P_{\delta_A,\F_2}(F)
\]
as $\Q$-graded bimodules over $(A^{\F_2}_{EF}(M_3),A^{\F_2}_{EF}(M_1))$.
\end{theorem}

A version of Theorem~\ref{thm:IntroMain1F2} holds over $\Z$; see Theorem~\ref{thm:GeneralCompositionWithSigns}.

\begin{example}\label{ex:PantsIntro}
    Among the morphisms in $\Cob^{\ext}$ are the closed and open $p$-tuples of pants as shown in Figure~\ref{fig:OpenClosedPants}. In many TQFTs, such cobordisms are assigned the $p$-fold tensor product functor on the monoidal category associated to the circle or the interval respectively. In terms of bimodules, if the circle or interval is assigned a Hopf algebra $H$, then one expects the corresponding $p$-tuple of pants cobordism to be assigned $H^{\otimes p}$ as a bimodule over $(H,H^{\otimes p})$ with right action by multiplication and left action by the coproduct of $H$. In our setting, we can ask which choices of $A \in \Q$ are such that the closed or open $p$-tuple of pants cobordisms get assigned this expected bimodule with the correct grading.
    \begin{itemize}
        \item The closed $p$-tuple of pants $\mathcal{P}_{p,\mathrm{closed}}$ has $h = p$. It has no components intersecting $S_-$ but not $S_+$, one component intersecting $S_+$ but not $S_-$, and no closed components. It has no $S_+$ intervals and $p+1$ $S_+$ circles. Thus,
        \[
        \delta_A(\mathcal{P}_{p,\mathrm{closed}}) = -Ap + A - (p+1)/2
        \]
        and we want this quantity to equal $-p$. The equation 
        \[
        p(-A+1/2) + (A-1/2) = 0
        \]
        holds for all $p$ if and only if $A = 1/2$; note that for a general open-closed cobordism $F$ we have
        \begin{align*}
        \delta_{1/2}(F) =& -h/2 - (1/2)\#\{\textrm{no-}S_+\textrm{ non-closed components of }F\} \\ 
        &+ (1/2)\#\{\textrm{no-}S_-\textrm{ non-closed components of }F\} \\
        & - (1/4) \#\{S_+\textrm{ intervals}\} - (1/2)\#\{S_+\textrm{ circles}\}.
        \end{align*}

        \item The open $p$-tuple of pants $\mathcal{P}_{p,\mathrm{open}}$ also has $h = p$. It has no components intersecting $S_-$ but not $S_+$, no components intersecting $S_+$ but not $S_-$, and no closed components. It has $p+1$ $S_+$ intervals and no $S_+$ circles. Thus,
        \[
        \delta_A(\mathcal{P}_{p,\mathrm{open}}) = -Ap + ((A-1)/2)(p+1).
        \]
        and we want this quantity to equal $-p$. The equation 
        \[
        p(-A/2+1/2) + (A-1)/2 = 0
        \]
        holds for all $p$ if and only if $A = 1$; note that for a general open-closed cobordism $F$ we have
        \begin{align*}
        \delta_{1}(F) = &-h + \#\{\textrm{no-}S_-\textrm{ non-closed components of }F\}  \\ 
        & + \#\{\textrm{closed components of }F\} - (1/2)\#\{S_+\textrm{ circles}\}.\\
        \end{align*}
    \end{itemize}
    Thus, the choices $A = 1/2$ and $A = 1$ in \eqref{eq:DeltaADef} are of particular interest. The first seems most natural for circle gluing and the connection with the existing literature on 3d non-semisimple TQFTs as in \cite{ManionDHASigns} (discussed further below); the second seems most natural for the connection between higher tensor products and the type of surface gluing along intervals that appears in \cite{ManionRouquier}.
\end{example}

As a corollary of Theorem~\ref{thm:IntroMain1F2}, we get a $1+1$-dimensional open-closed TQFT valued in algebras and bimodules. Since it is valued in algebras and bimodules rather than vector spaces and linear maps, we think of it conceptually as being part of the extended structure of a $2+1$-dimensional TQFT such as the 3d non-semisimple TQFTs discussed below. The proof of the following corollary is the same as \cite[proof of Corollary 1.3]{ManionDHASigns}. 

\begin{corollary}\label{cor:OpenClosedF2}
For any $A \in \Q$, the assignments 
\[
M \mapsto A^{\F_2}_{EF}(M) \quad \textrm{ and } \quad F \mapsto \Zb^P_{\delta_A(F),\F_2}(F)
\]
give a symmetric monoidal functor from $\Cob^{\ext}$ to the symmetric monoidal category $\Alg_{\F_2}$ of $\Z$-graded algebras over $\F_2$ and $\Q$-graded bimodules up to isomorphism.
\end{corollary}

\begin{figure}
    \centering
    \includegraphics[scale=0.9]{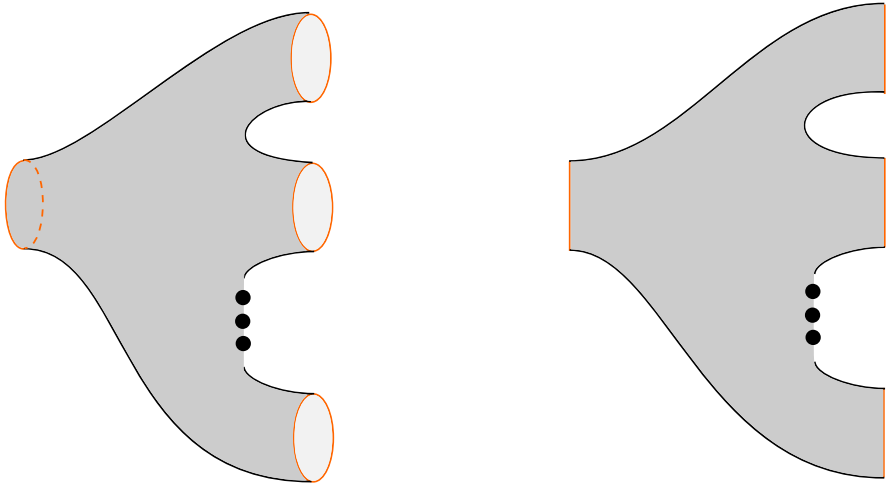}
    \caption{Left: closed $p$-tuple of pants, with $p$ circles on its right side and one on its left side. Right: open $p$-tuple of pants, with $p$ intervals on its right side and one on its left side.}
    \label{fig:OpenClosedPants}
\end{figure}

By contrast, when decategorifying the higher actions of \cite{ManionRouquier} one is led to spaces
\[
\Zb^{S_+}_{\delta,\F_2}(F) := \wedge^* H_1(F,S_+;\F_2)\{\delta(F)\}
\]
where $\delta$ is chosen from a larger parametrized family of functions $\Sc \to \Q$ described in \cite{ManionDHASigns}. If $\Cob^{\ext}_{\open}$ denotes the full subcategory of $\Cob^{\ext}$ on objects consisting only of intervals and no circles, then by \cite[Corollary 1.3]{ManionDHASigns}, the spaces $\Zb^{S_+}_{\delta,\F_2}(F)$ give a functor from $\Cob^{\ext}_{\open}$ into $\Z$-graded $\F_2$-algebras and $\Q$-graded bimodules. They seem to admit natural actions only of $U^{\F_2}(\psl(1|1)^+)$ and not of $U^{\F_2}(\psl(1|1))$. 

Let $V(0,0)^{\F_2}$ denote the two-dimensional $\Z$-graded module over $U^{\F_2}(\psl(1|1))$ given by the quotient $U^{\F_2}(\psl(1|1)) / (F)$; the notation is adapted from \cite[Section 2.3.2]{GeerYoung}. Note that $V(0,0)^{\F_2}$ is localized in $\Z$-degrees $-1$ and $0$. For $F$ with $p$ $S_+$ circles, we will show in Proposition~\ref{prop:SPlusFromP} that
\begin{equation}\label{eq:ZSPlusFromZP}
\Zb^{S_+}_{\delta_A,\F_2}(F) \cong \Zb^P_{\delta_A,\F_2}(F) \otimes_{(U^{\F_2}(\psl(1|1)))^{\otimes p}} (V(0,0)^{\F_2})^{\otimes p}.
\end{equation}
This formula suggests that the proper interpretation of $\Zb^{S_+}_{\delta_A,\F_2}(F)$ is the state space obtained from the larger state space $\Zb^P_{\delta_A,\F_2}(F)$ by labeling each $S_+$ boundary component with the representation $V(0,0)^{\F_2}$ of $U^{\F_2}(\psl(1|1))$. 

\begin{proposal}
We have the expected categorified actions of $U(\psl(1|1)^+)$ in \cite{ManionRouquier}, but not the expected categorified actions of $U(\psl(1|1))$, because the higher actions of \cite{ManionRouquier} are on categorifications of the state spaces $\Zb^{S_+}_{\delta_A,\F_2}(F)$ (with all actions of $U^{\F_2}(\psl(1|1))$ ``labeled away'') rather than on categorifications of the larger state spaces $\Zb^P_{\delta_A,\F_2}(F)$. To honestly get actions of $E$ and $F$ at the same time, one should first categorify $\Zb^P_{\delta_A,\F_2}(F)$, possibly by some adaptation or generalization of the strands algebra construction in bordered sutured Heegaard Floer homology in which the algebras have more basic idempotents whenever $\Zb^P_{\delta_A,\F_2}(F)$ is larger than $\Zb^{S_+}_{\delta_A,\F_2}(F)$.
\end{proposal}

\begin{remark}
    Even if one were to categorify $\Zb^P_{\delta_A,\F_2}(F)$ with its actions of both $E$ and $F$, these actions would still satisfy $EF + FE = 0$ at the decategorified level. The relation $EF + FE = 0$ is the $\mathfrak{gl}(1|1)$ relation $EF + FE = (K - K^{-1})/(q-q^{-1})$ when acting on representations with $K$-weight $1$ or $-1$, and one would hope to have categorifications of more general $K$-weight (e.g. the vector or defining representation of $U_q(\gl(1|1))$ with $K = q$) with actions of $E$ and $F$ satisfying the appropriate relations at the decategorified level. The connection with the 3d $\psl(1|1)$ Chern--Simons TQFT \cite{Mikhaylov} and Geer--Young's $\Dc^{q,\intt}$ TQFT \cite{GeerYoung}, in which the relation $EF + FE = 0$ arises for ``critical'' or non-generically decorated surfaces while relations with $EF + FE$ nonzero arise for generically decorated surfaces, could help see how to further generalize hypothetical categorifications of $\Zb^P_{\delta_A,\F_2}(F)$ to incorporate higher actions of both $E$ and $F$ with $EF + FE$ nonzero.
\end{remark}

\begin{remark}
Higher actions of both $E$ and $F$ on bordered strands algebras, satisfying $EF + FE = 0$, also appear in work of Ellis--Petkova--V{\'e}rtesi \cite{EPV} on Petkova--V{\'e}rtesi's tangle Floer homology \cite{PetkovaVertesi}. The presence of both $E$ and $F$ in this case seems to be a consequence of special symmetries that are enjoyed by the arc diagrams underlying the strands algebras in question but not by the arc diagrams for more general strands algebras. Tian \cite{TianUQ}, repurposing the tools of bordered sutured Heegaard Floer homology in a different way, also has actions of $E$ and $F$ satisfying the $U_q(\mathfrak{gl}(1|1))$ relations on certain algebras. While neither of these instances seem closely connected to what we discuss here, it would be interesting to pursue any connections if they exist.
\end{remark}

\begin{remark}
    In this paper we work with $U(\psl(1|1))$ acting on graded vector spaces. In some situations it is equivalent to work with $U(\pgl(1|1))$ acting on ordinary vector spaces such that the extra generator ``$H_2$'' of $U(\pgl(1|1))$ acts diagonalizably; then eigenspaces for $H_2$ stand in for summands of a graded vector space in different degrees. However, here we are especially concerned with functors given by tensor product with the algebra of a circle. If $M$ is a right module and $N$ is a left module over $U(\pgl(1|1))$, then $M \otimes_{U(\pgl(1|1))} N$ tensors the $H_2$ weight-$k$ subspace of $M$ with the $H_2$ weight-$k$ subspace of $N$ for each $k$, and the result has $H_2$ weight $k$. By contrast, if $M$ is a right graded module and $N$ is a left graded module over $U(\psl(1|1))$, then $M \otimes_{U(\psl(1|1))} N$ tensors the degree-$k$ subspace of $M$ with the degree-$l$ subspace of $N$ for each $(k,l)$, and the result has degree $k+l$. This second behavior is compatible with our gluing theorems while the first is not. Thus, the perspective taken here seems to favor $\psl(1|1)$ over $\pgl(1|1)$ in this sense.
\end{remark}

\noindent \textbf{Signs and parities.} If we want to pass from $\F_2$ to $\Q$ or $\Z$, we need to take signs and parities into account as well. The exterior algebra $\wedge^* H_1(F,P;\Z)$ is naturally a super abelian group with $\wedge^k$ in parity $k$ modulo $2$; for any given $F$ we may or may not reverse this parity, and we want our choices to be compatible with gluing.

We are not aware of a way to define a parity-shift function $\pi$ from the set $\Sc$ of morphisms in $\Cob^{\ext}$, additive under disjoint union, to $\Z/2\Z$ such that both interval and circle gluing theorems are satisfied. However, interesting parity-shift functions exist when restricting attention to certain subsets of $\Sc$. We first note that if we could choose $A$ such that $\delta_A(F)$ were integral for all $F$, then we could define the parity shift using $\delta_A(F)$ modulo 2, and gluing would be compatible with parity because it is compatible with $\delta_A$. However, there is no $A$ such that $\delta_A(F)$ is always an integer (see Proposition~\ref{prop:NoAWorksForAllF}). To proceed, we ask when $\delta_A(F)$ is integral in the special cases $A = 1/2$ and $A = 1$:
\begin{itemize}
    \item $\delta_{1/2}(F)$ is an integer if and only if the number of $S_+$ intervals of $F$ is equal modulo 4 to twice the number of boundary components of $F$ intersecting $S_-$ nontrivially (see Section~\ref{sec:GradingsParities}). In this case, write
    \[
    \pi_{1/2}(F) := \delta_{1/2}(F)
    \]
    modulo 2.

    \smallskip

    \item $\delta_1(F)$ is an integer if and only if the number of $S_+$ circles of $F$ is even. In this case, write
    \[
    \pi_1(F) := \delta_1(F)
    \]
    modulo 2.
\end{itemize}
Both of the above properties are preserved under disjoint union and under gluing an $S_+$ interval to an $S_+$ interval or an $S_+$ circle to an $S_+$ circle. 

\begin{definition}
    Let $\Cob^{\ext}_{\mathrm{closed}++}$ denote the subcategory of $\Cob^{\ext}$ whose morphisms have number of $S_+$ intervals equal modulo 4 to twice their number of boundary components intersecting $S_-$ nontrivially. Let $\Cob^{\ext}_{\mathrm{open}++}$ denote the subcategory of $\Cob^{\ext}$ whose morphisms have an even number of $S_+$ circles. Note that the ``closed sector'' of $\Cob^{\ext}$ (the usual 1+1-dimensional oriented cobordism category with $S_-$ always empty) is a subcategory of $\Cob^{\ext}_{\cpp}$, while the ``open sector'' of $\Cob^{\ext}$ (the full subcategory on objects with no circles) is a subcategory of $\Cob^{\ext}_{\opp}$.
\end{definition}

For any choice of $A \in \Q$ and $\pi \in \Z/2\Z$:
\begin{itemize}
    \item If $F$ is a morphism in $\Cob^{\ext}_{\cpp}$, the $\Q$-graded super abelian group
    \[
    \Zb^P_{\delta_A,\pi_{1/2}}(F) := \left( \Z^{0|1} \right)^{\otimes \pi_{1/2}(F)} \otimes \wedge^* H_1(F,P) \{\delta_A(F)\}
    \]
    canonically has the structure of a bimodule over $(A_{EF}(M_2), A_{EF}(M_1))$.
    \item If $F$ is a morphism in $\Cob^{\ext}_{\opp}$, the $\Q$-graded super abelian group
    \[
    \Zb^P_{\delta_A,\pi_{1}}(F) := \left( \Z^{0|1} \right)^{\otimes \pi_{1}(F)} \otimes \wedge^* H_1(F,P) \{\delta_A(F)\}
    \]
    canonically has the structure of a bimodule over $(A_{EF}(M_2), A_{EF}(M_1))$.
\end{itemize} 

Theorem~\ref{thm:GeneralCompositionWithSigns} (the $\Z$ version of Theorem~\ref{thm:IntroMain1F2}) gives us the following corollary.

\begin{corollary}
    Let $\SAlg_{\Z}$ denote the symmetric monoidal category of $\Z$-graded super rings and $\Q$-graded bimodules up to isomorphism.
\begin{itemize}
    \item For any $A \in \Q$, the assignments
    \[
    M \mapsto A_{EF}(M) \quad \textrm{ and } \quad F \mapsto \Zb^P_{\delta_A, \pi_{1/2}}(F)
    \]
    give a symmetric monoidal functor from $\Cob^{\ext}_{\cpp}$ to $\SAlg_{\Z}$.

    \item For any $A \in \Q$, the assignments
    \[
    M \mapsto A_{EF}(M) \quad \textrm{ and } \quad F \mapsto \Zb^P_{\delta_A, \pi_1}(F)
    \]
    give a symmetric monoidal functor from $\Cob^{\ext}_{\opp}$ to $\SAlg_{\Z}$.
\end{itemize}
\end{corollary}

The proof is the same as \cite[proof of Corollary 1.3]{ManionDHASigns}.
    
\smallskip

\noindent \textbf{Non-semisimple TQFT.} Recently there has been considerable interest in non-semisimple analogues of 3d Witten--Reshetikhin--Turaev TQFTs (see e.g. \cite{ADO, CGPM, Mikhaylov, BCGPM, GPV, AGPS, GPPV, GHNPPS,  CGP-NSS, Jagadale, GeerYoung}). In particular, some of these constructions \cite{Mikhaylov,BCGPM,AGPS,GeerYoung} arise from the quantum representation theory of the Lie superalgebra $\mathfrak{gl}(1|1)$ or its relatives and have connections with the Alexander polynomial and Reidemeister torsion. A general mechanism for defining non-semisimple 3d TQFTs is given by De Renzi \cite{DeRenzi}, who uses a universal construction to define a type of 1+1+1 extended TQFT starting with data that he calls a relative modular category $\Cc$. In particular, $\Cc$ comes with a decomposition $\Cc = \oplus_{g \in G} \, \Cc_g$ where $G$ is some abelian group, and the data also includes another abelian group $Z$ such that the state spaces of the theory on (decorated) surfaces are $Z$-graded vector spaces. 

In \cite[Theorem 2.23 with $\hbar = \pi i / 2$ and $q = i$]{GeerYoung} Geer and Young define a $\C$-linear\footnote{When discussing non-semisimple TQFT we will work over $\C$ rather than $\Z$.} relative modular category $\Dc^{q,\intt}$ with $G = \C/((2\pi i / \hbar) \Z) = \C/(4\Z)$ and $Z = \Z \times \Z \times \Z/2\Z$. They propose that the TQFT associated to $\Dc^{q,\intt}$ by De Renzi's construction \cite{DeRenzi} is a (homologically truncated, non-derived) mathematical realization of the main subject of Mikhaylov's paper \cite{Mikhaylov}, referred to by Mikhaylov as the $\mathfrak{psl}(1|1)$ Chern-Simons TQFT. We will call the non-extended version of this TQFT $\Zb^{GY}$ and the extended version $\Zb^{GY}_{\ext}$. 

To a disjoint union of $p$ circles each decorated by $0 \in G$, $\Zb^{GY}_{\ext}$ assigns a category enriched in $Z$-graded vector spaces, which can be viewed as the idempotent completion of the $p$-fold ordinary tensor power (as in \cite[Section 1.4]{KellyEnriched}) of the $Z$-graded category of $Z$-graded projective modules over $U(\psl(1|1))$. The $p$-fold tensor product
\[
P(0,0)_{\overline{0}} \otimes \cdots \otimes P(0,0)_{\overline{0}}
\]
naturally gives an object of the category, where $P(0,0)_{\overline{0}}$ denotes $U(\psl(1|1))$ with first component of the $Z$-grading identically zero, second component of the $Z$-grading given by the usual grading on $U(\psl(1|1))$ as a $\Z$-graded module over itself with $\deg(E) = -1$ and $\deg(F) = 1$, and third component of the $Z$-grading given by the parity on the super vector space $U(\psl(1|1))$. See \cite[Section 2.3.3]{GeerYoung}.

Now, let $F$ be a zero-decorated 2d cobordism from $M_1$ (consisting of $p_1$ zero-decorated circles) to $M_2$ (consisting of $p_2$ zero-decorated circles), and assume that each component of $F$ intersects $M_1$ nontrivially, which (given these decorations) amounts to the admissibility condition defined in \cite[Section 2.3]{DeRenzi} for $F$ as a cobordism from $M_1$ to $M_2$. The TQFT $\Zb^{GY}_{\ext}$ assigns to $F$ a functor between the above enriched categories, and we can evaluate this functor on the object corresponding to $\left( P(0,0)_{\overline{0}} \right)^{\otimes p_1}$. We get an object of the category associated to $M_2$, and this category admits a canonical functor to the category of $Z$-graded projective modules over $\left( U(\psl(1|1)) \right)^{\otimes p_2}$. Applying this additional functor, we get a $Z$-graded super vector space $\Zb^{GY}_{\ext}(F)$ (slightly abusing notation) with a left action of 
\[
\left( U(\psl(1|1)) \right)^{\otimes p_2} = A_{EF}(M_2),
\]
projective as a left module. Furthermore, by applying the same sequence of functors to morphisms from $\left( P(0,0)_{\overline{0}} \right)^{\otimes p_1}$ to itself, $\Zb^{GY}_{\ext}(F)$ also has a right action of 
\[
\left( U(\psl(1|1)) \right)^{\otimes p_1} = A_{EF}(M_1).
\] 
Overall, there is a bimodule structure over $(A_{EF}(M_2), A_{EF}(M_1))$, projective as a left module, on $\Zb^{GY}_{\ext}(F)$. We can also discard the first component of the grading by $Z = \Z \times \Z \times \Z/2\Z$, which will always be zero in this setting, and the third component which is captured by the super vector space structure. Thus, we will view $\Zb^{GY}_{\ext}(F)$ as a $\Z$-graded bimodule.

It turns out (see Proposition~\ref{prop:Projective}) that the spaces $\Zb^P_{\delta,\pi}(F)$ are projective as left modules over $A_{EF}(M_2)$ if and only if each component of $F$ intersects the incoming boundary $M_1$, i.e. exactly when De Renzi's admissibility condition is satisfied. The surfaces $F$ in question are morphisms in $\Cob^{\ext}_{\cpp}$, so they have $\delta_{1/2}(F) \in \Z \subset \Q$. For such $F$ we make the following conjecture.

\begin{conjecture}\label{conj:WedgeEqualsGY}
    If each component of $F$ intersects $M_1$, then as $\Z$-graded bimodules over $(A_{EF}(M_2),A_{EF}(M_1))$ we have
    \[
    \Zb^{GY}_{\ext}(F) \cong \Zb^P_{\delta_{1/2},\pi_{1/2}}(F).
    \]
\end{conjecture}

Note that Conjecture~\ref{conj:WedgeEqualsGY} uses the same grading and parity shifts as does \cite[Conjecture 1.6]{ManionDHASigns} about $\Zb^{S_+}$. For genus zero connected surfaces $F$, De Renzi discusses the graded dimensions of state spaces in \cite[Section 7.5]{DeRenzi}, and our degree and parity shifts for $\Zb^P_{\delta_{1/2},\pi_{1/2}}(F)$ are compatible with this discussion. Going beyond the graded dimensions and looking at the bimodule structure, in Example~\ref{prop:Pants} we will show that for the closed $p$-tuple of pants cobordism $F$ from $(S^1)^{\sqcup p}$ to $S^1$, the bimodule $\Zb^P_{\delta_{1/2},\pi_{1/2}}(F)$ can be identified with $(U(\psl(1|1)))^{\otimes p}$ with right action of $(U(\psl(1|1)))^{\otimes p}$ by multiplication and left action of $U(\psl(1|1))$ induced by the coproduct $\Delta(E) = E \otimes 1 + 1 \otimes E$ and $\Delta(F) = F \otimes 1 + 1 \otimes E$ on $U(\psl(1|1))$. It follows that the functor $(\Zb^P_{\delta_{1/2},\pi_{1/2}}(F) \otimes_{(U(\psl(1|1)))^{\otimes p}} -)$ gives the $p$-fold tensor product of representations of $U(\psl(1|1))$, in line with \cite[Proposition 7.3]{DeRenzi} for $\Zb^{GY}_{\ext}(F)$.

\begin{remark}
While the TQFT $\Zb^{GY}$ is functorial on decorated cobordisms as in \cite{GeerYoung} (roughly: 3d cobordisms without corners, equipped with colored ribbon graphs), $\wedge^* H_1(F,S_+)$ is a decategorification of Zarev's bordered sutured Heegaard Floer invariants of surfaces and thus should be functorial under certain ``sutured cobordisms'' (roughly: 3d cobordisms with corners). While Geer--Young define their TQFT using the universal construction applied to decorated cobordisms, bordered sutured Heegaard Floer homology and its decategorification can also be seen as arising from a universal construction applied to sutured cobordisms (``Zarev caps''), by the ideas of \cite{ZarevJoiningGluing}.

It seems likely that a subset of decorated cobordisms can be identified with a subset of sutured cobordisms; the decorated cobordisms allow more general colorings while the sutured cobordisms allow more general topology. It would be interesting to make this identification and extend it to define a type of cobordism jointly generalizing decorated cobordisms and sutured cobordisms (possibly: sutured cobordisms equipped with some more general type of coloring data). Then one would hope that $\Zb^{GY}$ is functorial under these more general cobordisms and that a theory can be defined extending $\Zb^{GY}_{\ext}$ to surfaces with corners and sutured cobordisms, and extending decategorified bordered sutured Heegaard Floer theory from colors related to $\mathfrak{psl}(1|1)$ to more general colors related to Geer--Young's $U^E_q(\mathfrak{gl}(1|1))$ (see the discussion in \cite[Section 6.1]{GeerYoung}). 
\end{remark} 

\begin{remark}
The results of this paper were also motivated in a different direction by open questions in bordered Heegaard Floer homology. Mikhaylov's $\mathfrak{psl}(1|1)$ Chern-Simons theory is meant to recover the Turaev torsion for closed 3-manifolds; in turn, this torsion is categorified by the sophisticated $HF^+$ and $HF^-$ versions of Heegaard Floer homology, while to this point bordered Heegaard Floer homology has been largely limited to the setting of the simpler version $\widehat{HF}$. For example, one cannot recover the interesting Ozsv{\'a}th--Szab{\'o} mixed invariants for smooth 4-manifolds from $\widehat{HF}$; one needs $HF^+$ and $HF^-$. For genus-zero surfaces there is some work that goes beyond the $\widehat{HF}$ setting (e.g. Ozsv{\'a}th--Szab{\'o}'s bordered HFK) and in genus one Lipshitz--Ozsv{\'a}th--Thurston have given talks on a bordered version \cite{LOTMinus} of $HF^-$ that has not yet appeared in the literature (the relevant algebra is introduced in \cite{LOTMinusAlg}). In general, though, it is a major open problem to extend bordered Heegaard Floer techniques so that $HF^+$ and $HF^-$ can be recovered by cutting 3-manifolds along surfaces. One could hope to approach this problem by thoroughly understanding the decategorified level first, and in particular understanding a suitable TQFT approach to the Turaev torsion such as $\mathfrak{psl}(1|1)$ Chern--Simons theory, then trying to categorify everything and recover $HF^+$ and $HF^-$ by combining methods from bordered Heegaard Floer homology and 3d non-semisimple TQFTs. The current paper attempts to take a step toward this goal.
\end{remark}

\smallskip

\noindent \textbf{Organization.} In Section~\ref{sec:AlgebraActions} we review some preliminary definitions, define algebra actions on the spaces $\wedge^* H_1(F,P)$ for both interval and circle components of $S_+$, discuss when these actions give projective modules, and relate $\wedge^* H_1(F,S_+)$ to $\wedge^* H_1(F,P)$. In Section~\ref{sec:GluingTheorem} we prove our main results, and in Section~\ref{sec:GradingsParities} we explain our choices of degree and parity shifts in more detail.

\smallskip

\noindent \textbf{Acknowledgments.} The author would like to thank Bojko Bakalov, Sergei Gukov, Corey Jones, Rapha{\"e}l Rouquier, and Matthew Young for helpful conversations. The author was partially supported by NSF grant number DMS-2151786.

\section{Algebra actions for intervals and circles}\label{sec:AlgebraActions}

\subsection{Sutured surfaces and open-closed cobordisms}

We recall the definition we will use for sutured surfaces following the presentation in \cite{ManionDHASigns}.

\begin{definition}[cf. Definition 1.2 of \cite{Zarev}, Definition 2.1 of \cite{ManionDHASigns}]
A sutured surface consists of the data $(F,\Lambda,S_+,S_-,\ell)$ where:
\begin{itemize}
\item $F$ is a compact oriented surface, possibly with boundary ($F$ can be disconnected and is allowed to have closed components);
\item $\Lambda$ is a choice of some even number of points (possibly none) in each boundary component of $F$;
\item The components of $\partial F \setminus \Lambda$ are labeled as being in either $S_+$ or $S_-$, in alternating fashion across the points of $\Lambda$. Components of $\partial F$ with no points of $\Lambda$ are either $S_+$ circles or $S_-$ circles; the rest of the components of $S_+$ and $S_-$ are closed intervals.
\item $\ell$ consists of a labeling of each component of $S_+$ as ``incoming'' or ``outgoing,'' as well as an ordering on the set of incoming $S_+$ components and an ordering on the set of outgoing $S_+$ components.
\end{itemize}
\end{definition}

We will usually refer to a sutured surface as $F$ and suppress mention of the rest of the data. We can view a sutured surface $F$ as a morphism in the 1+1-dimensional open-closed cobordism category $\Cob^{\ext}$ defined by Lauda--Pfeiffer in \cite{LaudaPfeiffer}. The source of this morphism is the incoming part $M_1$ of $S_+$ (after orientation reversal) and the target is the outgoing part $M_2$, so that $S_+ = M_2 \sqcup (-M_1)$. The non-gluing boundary of $F$ is $S_-$. Sutured surfaces corresponding to a pair of composable morphisms in $\Cob^{\ext}$ are shown in \cite[Figure 2]{ManionDHASigns}.

\subsection{Actions on larger state spaces}

Choose a finite subset $P$ of $S_+$ consisting of one point in each component (interval or circle) of $S_+$. For any $\delta \in \Q$ and $\pi \in \Z/2\Z$ we will define superalgebra actions on the super abelian group
\[
\Zb^P_{\delta,\pi}(F) := (\Z^{0|1})^{\otimes \pi(F)} \otimes \wedge^* H_1(F,P) \{\delta(F)\}.
\]
Specifically, we will have actions of $\Z[E]/(E^2)$ for interval components of $S_+$ and actions of $\Z\langle E,F \rangle/(E^2,F^2,EF+FE)$ for circle components of $S_+$. The actions will be left actions for outgoing components of $S_+$ and right actions for incoming components of $S_+$. Write $\varepsilon_F$ for the standard basis element of $(\Z^{0|1})^{\otimes \pi(F)}$.

\begin{definition}[cf. Definition 2.4 of \cite{ManionDHASigns}]\label{def:EActions}
Let $(F,\Lambda,S_+,S_-,\ell)$ be a sutured surface and let $X$ be an outgoing interval or circle component of $S_+$. Choose a finite set of points $P$ as above. For any $\delta \in \Q$ and $\pi \in \Z/2\Z$, we define
\[
E = E_I \colon \Zb^P_{\delta,\pi}(F) \to \Zb^P_{\delta,\pi}(F)
\]
as follows.
\begin{itemize}
\item Define $\phi_X \colon H_1(F,P) \to \Z$ to be the composition
\[
\phi_X := H_1(F,P) \xrightarrow{\partial} H_0(P) \xrightarrow{[p_X]^* \cdot -} \Z
\]
where $[p_X]^*$ is the class in $H^0(P)$ dual to the homology class in $H_0(P)$ of the unique point $p_X$ of $P$ contained in $X$.

\item For $k \geq 1$, define $\Phi_X \colon T^k H_1(F,P) \to T^{k-1} H_1(F,P)$ by
\begin{align*}
\Phi_X &:= \sum_{i=1}^k \bigg( T^{i-1}H_1(F,P) \otimes_{\Z} H_1(F,P) \otimes_{\Z} T^{k-i} H_1(F,P) \\
& \xrightarrow{(-1)^{i-1} \id_{T^{i-1}H_1(F,P)} \otimes \phi_X \otimes \id_{T^{i-1}H_1(F,P)}} T^{i-1}H_1(F,P) \otimes_{\Z} \Z \otimes_{\Z} T^{k-i} H_1(F,P). \bigg)
\end{align*}

\item Because of the sign $(-1)^{i-1}$ in the above definition, we get an induced map
\[
\overline{\Phi_X} \colon \wedge^k H_1(F,P) \to \wedge^{k-1} H_1(F,P). 
\]
Define
\[
\overline{\Phi_X}' \colon (\Z^{0|1})^{\otimes \pi(F)} \otimes \wedge^k H_1(F,P) \{\delta(F)\} \to (\Z^{0|1})^{\otimes \pi(F)} \otimes \wedge^{k-1} H_1(F,P) \{\delta(F)\}
\]
by
\[
\overline{\Phi_X}'(\varepsilon_F \otimes \omega) := (-1)^{\pi(F)} \varepsilon_F \otimes \overline{\Phi_X}(\omega).
\]

\item Let $E$ be the sum of the maps $\overline{\Phi_X}'$ over all $k \geq 1$. $E$ is an odd map and the sign $(-1)^{i-1}$ in the definition of $\Phi_X$ ensures that $E^2 = 0$.
\end{itemize}
Define $E$ similarly when $X$ is an incoming interval or circle component of $S_+$, except that:
\begin{itemize}
\item Instead of $(-1)^{i-1}$ we have $(-1)^{k-i}$ as the sign in the definition of $\Phi_X$;
\item We define $\overline{\Phi_X}'$ by
\[
\overline{\Phi_X}'(\varepsilon_F \otimes \omega) := \varepsilon_F \otimes \overline{\Phi_X}(\omega),
\]
without a sign of $(-1)^{\pi(F)}$.
\end{itemize}
\end{definition}

\begin{remark}
    In \cite{ManionDHA,ManionDHASigns}, we could have defined actions of $\Z[E]/(E^2)$ on $\Zb^{S_+}_{\delta,\pi}(F)$ for circle components of $S_+$ like we do in Definition~\ref{def:EActions}. However, these actions do not seem very motivated; if we think of $\Zb^{S_+}_{\delta,\pi}(F)$ as being obtained from $\Zb^{P}_{\delta,\pi}(F)$ as in Proposition~\ref{prop:SPlusFromP}, then by labeling $S_+$ circles we should be ``using up'' the algebra action on each $S_+$ circle without leaving any residual action. Correspondingly, $\Zb^{S_+}_{\delta,\pi}(F)$ does not appear to admit a gluing theorem when gluing along circles.
\end{remark}

The following definition is where we finally get $F$ endomorphisms to go along with the $E$ endomorphisms for circles.

\begin{definition}\label{def:FActions}
Let $(F,\Lambda,S_+,S_-,\ell)$ be a sutured surface and let $C$ be a circle component of $S_+$. Give $C$ the orientation it has as (part of) an object of $\Cob^{\ext}$; in other words, give $C$ the boundary orientation induced from $F$ if $C$ is outgoing, and give $C$ the reverse of this orientation if $C$ is incoming. Choose a finite set of points $P$ as above. For any $\delta \in \Q$ and $\pi \in \Z/2\Z$, define
\[
F = F_C \colon \Zb^P_{\delta,\pi}(F) \to \Zb^P_{\delta,\pi}(F)
\]
to send 
\[
\varepsilon_F \otimes \omega \mapsto (-1)^{\pi(F)} \varepsilon_F \otimes ([C] \wedge \omega)
\]
if $C$ is outgoing and
\[
\varepsilon_F \otimes \omega \mapsto \varepsilon_F \otimes (\omega \wedge [C])
\]
if $C$ is incoming.
\end{definition}

\begin{proposition}\label{prop:Relations}
The following relations hold:
\begin{itemize}
\item If $I$ is an interval component of $S_+$, then the endomorphism $E = E_I$ from Definition~\ref{def:EActions} satisfies $E^2 = 0$.
\item If $C$ is a circle component of $S_+$, then the endomorphisms $E = E_C$ from Definition \ref{def:EActions} and $F = F_C$ from Definition~\ref{def:FActions} satisfy $E^2 = 0$, $F^2 = 0$, and $EF + FE = 0$.
\item Endomorphisms $E$ or $F$ for any pair of distinct incoming components of $S_+$ anticommute, and the same is true for any pair of distinct outgoing components. Any pair of endomorphisms commute if one comes from an incoming component and the other comes from an outgoing component.
\end{itemize}
\end{proposition}

\begin{proof}
The relations $E^2 = 0$ in the interval and circle case follow from the signs $(-1)^{i-1}$ or $(-1)^{k-i}$ in Definition~\ref{def:EActions}. The relation $F^2 = 0$ in the circle case follows from $[C] \wedge [C] = 0$.

Let $C$ be a circle component of $S_+$, so that we have two endomorphisms $E$ and $F$ corresponding to $C$. Informally, if $C$ is outgoing then $E$ acts by moving inward from the left (picking up a sign when crossing each factor) until it reaches various wedge factors with nontrivial boundary at the point $p_C$, and summing (with signs) over each way to remove of one of these wedge factors. Acting with $F$ before $E$ will ensure that each term in the subsequent action of $E$ will have to cross an extra factor $[C]$ compared with acting with $E$ before $F$. It follows that $EF + FE = 0$. The anticommutativity claim for two incoming components of $S_+$ and the commutativity claim for one incoming and one outgoing component of $S_+$ follow by similar arguments.
\end{proof}

For an object $M$ of $\Cob^{\ext}$, let $A_{EF}(M)$ denote the tensor product, in order, of superalgebras $\Z[E]/(E^2)$ for interval components of $M$ and $\Z\langle E,F \rangle / (E^2, F^2, EF + FE)$ for circle components of $M$. If $F \colon M_1 \to M_2$ is a morphism in $\Cob^{\ext}$, Proposition~\ref{prop:Relations} gives $\Zb^P_{\delta,\pi}(F)$ the structure of a bimodule over $(A_{EF}(M_2),A_{EF}(M_1))$. If we want, we can also tensor everything with $\F_2$ and forget signs and super structures. We will denote the resulting bimodule as $\Zb^P_{\delta,\F_2}(F)$ since $\pi$ is irrelevant over $\F_2$.

\begin{example}
    Figure~\ref{fig:EAndFActions} shows, for a particular sutured surface $F$ and a particular choice of $S_+$ circle $C$ and $S_+$ interval $I$ in its boundary, the actions of $F_C$, $E_C$, and $E_I$ on an element $\alpha_1 \wedge \alpha_2 \in \Zb^{S_+}_{\delta,\F_2}(F)$.
 \end{example}

\begin{figure}
    \centering
    \includegraphics[scale=0.7]{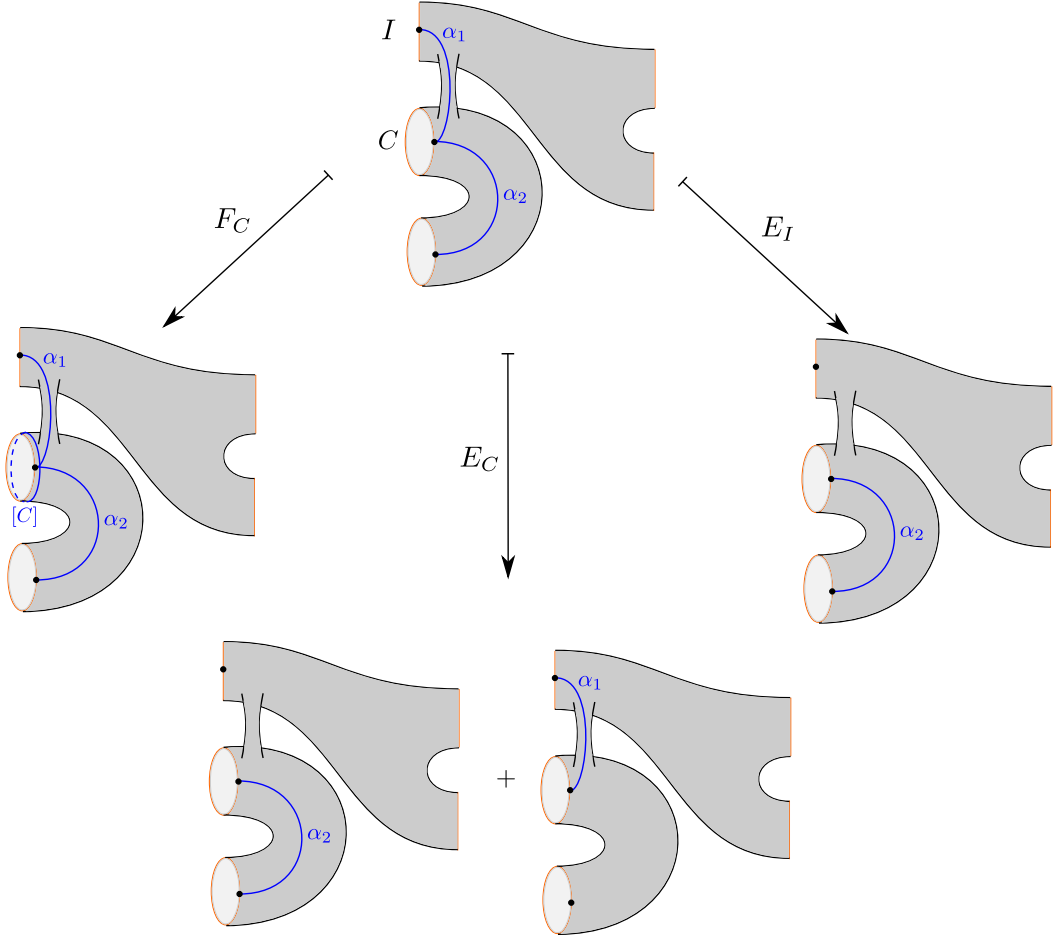}
    \caption{Actions of $F$ for circles and $E$ for circles and intervals (the coefficients are taken in $\F_2$ for simplicity).}
    \label{fig:EAndFActions}
\end{figure}

\subsection{Defining bases from collections of arcs and circles}\label{sec:ChoosingBases}

The proofs of the gluing results in \cite{ManionDHA,ManionDHASigns} make use of certain choices of basis for $H_1(F,S_+)$; the proofs here will use bases for $H_1(F,P)$ similarly. As in \cite[proof of Lemma 4.1]{ManionDHA} and \cite[proof of Lemma 3.1]{ManionDHASigns}, given a sutured surface $F$, we can choose a homeomorphism (preserving the sutured data) between $F$ and a finite disjoint union of ``standard'' sutured surfaces such as the one shown in Figure~\ref{fig:LargeBasis}. Define a collection of oriented arcs and circles in each standard sutured surface, e.g. the set of blue arcs and circles in Figure~\ref{fig:LargeBasis}, as follows.
\begin{itemize}
\item In each of the ``handles,'' take two oriented circles as in Figure~\ref{fig:LargeBasis}.
\item For all the boundary components of $F$ (whether or not they intersect $S_-$ nontrivially), except for one chosen component, take an oriented circle around the boundary component.
\item Choose a connected acyclic directed graph $\Gamma_F$ embedded in $F$ with vertex set $P$.
\end{itemize}
The circles and edges of $\Gamma_F$ give a basis for $H_1(F,P)$, so their wedge products give a basis for $\wedge^* H_1(F,P)$. Applying $\varepsilon_F \otimes -$ to these wedge products, we get a basis for $(\C^{0|1})^{\otimes \pi(F)} \otimes \wedge^* H_1(F,P)\{\delta(F)\}$.

\begin{figure}
    \centering
    \includegraphics[scale=0.8]{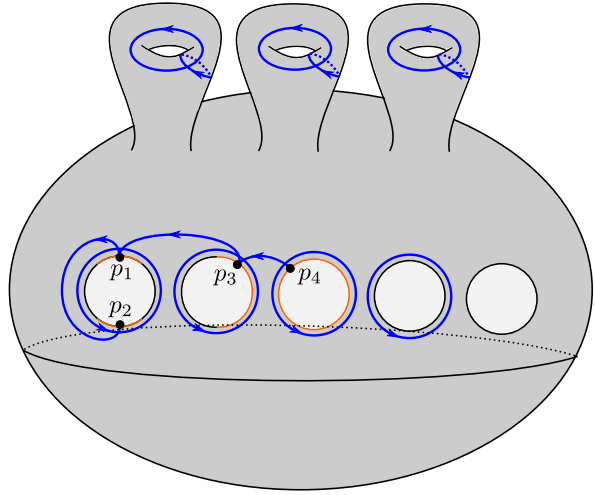}
    \caption{Oriented arcs and circles forming a basis of $H_1(F,P)$, whose products form a basis for $\wedge^* H_1(F,P)$. The set $P$ consists of the points labeled $p_1$, $p_2$, $p_3$, and $p_4$. The subset $S_+$ of $\partial{F}$ is drawn in orange, while $S_-$ is drawn in black. When compared with \cite[Figure 16]{ManionDHA}, we have an extra blue circle here.}
    \label{fig:LargeBasis}
\end{figure}

\subsection{Tensor products} Here we discuss an important family of examples, the closed and open $p$-tuples of pants from Example~\ref{ex:PantsIntro}, in more detail. In fact, for the open $p$-tuple of pants $\Pc_{\mathrm{open}}$, \cite[Proposition 2.7]{ManionDHASigns} applies basically unchanged; we have $H_1(F,P) \cong H_1(F,S_+)$ because $\Pc_{\mathrm{open}}$ has no $S_+$ circles, and the grading shifts of this paper were dealt with in Example~\ref{ex:PantsIntro}. It follows that
\[
\Zb^P_{\delta_{1},\pi_1}(\Pc_{\mathrm{open}}) \cong (\Z[E]/(E^2))^{\otimes p}
\]
as $\Z$-graded bimodules over $(\Z[E]/(E^2), (\Z[E]/(E^2))^{\otimes p})$ where $(\Z[E]/(E^2))^{\otimes p}$ acts on the right by multiplication and the left action of $\Z[E]/(E^2)$ is induced by the coproduct
\[
\Delta(E) = E \otimes 1 + 1 \otimes E.
\]

Let $\Pc_{\mathrm{closed}}$ denote the closed $p$-tuple of pants. In Figure~\ref{fig:ClosedPantsBasis}, label the arcs as $e_1,\ldots,e_p$ from top to bottom and the circles as $\sigma_1,\ldots,\sigma_p$ from top to bottom. Orient them as indicated in Figure~\ref{fig:ClosedPantsBasis}:
\begin{itemize}
    \item $e_p$ is oriented from the outgoing boundary to the incoming boundary, $e_{p-1}$ is oriented from the incoming boundary to the outgoing boundary, and so on in alternating fashion;
    \item $\sigma_p$ is given the opposite of the boundary orientation of its corresponding incoming $S_+$ circle, $\sigma_{p-1}$ is given the boundary orientation of its corresponding incoming $S_+$ circle, and so on in alternating fashion.
\end{itemize}

\begin{figure}
    \centering
    \includegraphics[scale=0.8]{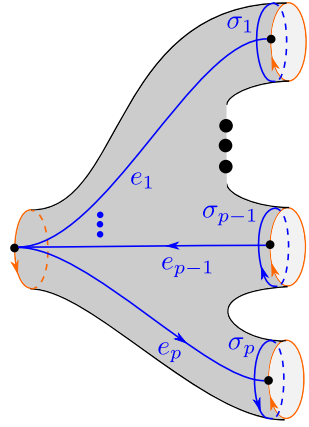}
    \caption{Orientations on arcs $e_1,\ldots,e_p$ and circles $\sigma_1,\ldots,\sigma_p$ giving a basis for $H_1(\Pc_{\mathrm{closed}}, P)$. The orientations shown on the $S_+$ boundary circles of $\Pc_{\mathrm{closed}}$ are the induced boundary orientations without reversal.}
    \label{fig:ClosedPantsBasis}
\end{figure}

We have the following analogue of \cite[Proposition 2.7]{ManionDHASigns} for the closed $p$-tuple of pants. 
\begin{proposition}\label{prop:Pants}
    Define a map 
    \[
    \Zb^P_{\delta_{1/2},\pi_{1/2}}(\Pc_{\mathrm{closed}}) \to (\Z\langle E,F \rangle / (E^2, F^2, EF + FE ) )^{\otimes p}
    \]
    by sending the basis element 
    \[
    \varepsilon_{\Pc_{\mathrm{closed}}} \otimes (e_1^{\delta_1} \wedge \sigma_1^{\delta'_1}) \wedge \cdots \wedge (e_p^{\delta_p} \wedge \sigma_p^{\delta'_p})
    \]
    of $\Zb^P_{\delta_{1/2},\pi_{1/2}}(\Pc_{\mathrm{closed}})$ (where $\delta_i, \delta'_i \in \{0,1\}$) to the element
    \[
    (E^{1 - \delta_1} F^{\delta'_1}) \otimes \cdots \otimes (E^{1 - \delta_p} F^{\delta'_p})
    \]
    of $(\Z\langle E,F \rangle / (E^2, F^2, EF + FE ) )^{\otimes p}$. This map is an isomorphism of $\Z$-graded bimodules over 
    \[
    (\Z\langle E,F \rangle / (E^2, F^2, EF + FE ), \, (\Z\langle E,F \rangle / (E^2, F^2, EF + FE ))^{\otimes p})
    \]
    where $(\Z\langle E,F \rangle / (E^2, F^2, EF + FE ))^{\otimes p}$ acts on the right by multiplication and the left action of $\Z\langle E,F \rangle / (E^2, F^2, EF + FE )$ is induced by the coproduct
\[
\Delta(E) = E \otimes 1 + 1 \otimes E, \qquad \Delta(F) = F \otimes 1 + 1 \otimes F.
\]
\end{proposition}

\begin{proof}
    By Example~\ref{ex:PantsIntro} and the definition of $\pi_{1/2}$ as the parity of $\delta_{1/2}$, the map respects $\Z$-grading and parity. It gives a bijection on basis elements, so we need to show it is compatible with left multiplication by $\Z\langle E,F \rangle / (E^2, F^2, EF + FE )$ and right multiplication by $(\Z\langle E,F \rangle / (E^2, F^2, EF + FE ))^{\otimes p}$.

    We start with left multiplication by $E$. Since $\pi_{1/2}(\Pc_{\mathrm{closed}}) = p$, the product
    \[
    E \cdot (\varepsilon_{\Pc_{\mathrm{closed}}} \otimes (e_1^{\delta_1} \wedge \sigma_1^{\delta'_1}) \wedge \cdots \wedge (e_p^{\delta_p} \wedge \sigma_p^{\delta'_p})),
    \]
    is given by
    \begin{equation}\label{eq:EActingOnLeftWedge}
    \begin{aligned}
    &(-1)^{p} \sum_{i: \delta_i = 1}(-1)^{\delta_1 + \delta'_1 + \cdots + \delta_{i-1} + \delta'_{i-1}} (-1)^{p-i+1} \\
    &\cdot \varepsilon_{\Pc_{\mathrm{closed}}} \otimes (e_1^{\delta_1} \wedge \sigma_1^{\delta'_1}) \wedge \cdots \wedge (1 \wedge \sigma_i^{\delta'_i}) \wedge \cdots \wedge (e_p^{\delta_p} \wedge \sigma_p^{\delta'_p});
    \end{aligned}
    \end{equation}
    the factor $(-1)^{p-i+1}$ comes from our choice of orientation on $e_i$. Meanwhile,
    \[
    E \cdot ((E^{1 - \delta_1} F^{\delta'_1}) \otimes \cdots \otimes (E^{1 - \delta_p} F^{\delta'_p}))
    \]
    is given by
    \begin{equation}\label{eq:EActingOnLeftTensor}
    \sum_{i : \delta_i = 1} (-1)^{(1-\delta_1) + \delta'_1 + \cdots + (1 - \delta_{i-1}) + \delta'_{i-1}} ((E^{1 - \delta_1} F^{\delta'_1}) \otimes \cdots \otimes (E F^{\delta'_i}) \otimes  \cdots \otimes (E^{1 - \delta_p} F^{\delta'_p})).
    \end{equation}
    The expressions \eqref{eq:EActingOnLeftWedge} and \eqref{eq:EActingOnLeftTensor} correspond under our map, so our map respects left multiplication by $E$.

    Next, to compute
    \begin{equation}\label{eq:FActingOnLeftWedgeStart}
    F \cdot (\varepsilon_{\Pc_{\mathrm{closed}}} \otimes (e_1^{\delta_1} \wedge \sigma_1^{\delta'_1}) \wedge \cdots \wedge (e_p^{\delta_p} \wedge \sigma_p^{\delta'_p})),
    \end{equation}
    we note that if $[C]$ is the homology class of the outgoing circle of $\Pc_{\mathrm{closed}}$, with boundary orientation, then
    \[
    [C] = \sigma_p - \sigma_{p-1} + \cdots + (-1)^{p-1} \sigma_1.
    \]
    Thus, \eqref{eq:FActingOnLeftWedgeStart} equals
    \begin{equation}\label{eq:FActingOnLeftWedgeEnd}
    \begin{aligned}
    &(-1)^p \sum_{i: \delta'_i = 0} (-1)^{\delta_1 + \delta'_1 + \cdots + \delta_{i-1} + \delta'_{i-1} + \delta_i} (-1)^{p-i} \\
    &\cdot \varepsilon_{\Pc_{\mathrm{closed}}} \otimes (e_1^{\delta_1} \wedge \sigma_1^{\delta'_1}) \wedge \cdots \wedge (e_i^{\delta_i} \wedge \sigma_i) \wedge \cdots \wedge (e_p^{\delta_p} \wedge \sigma_p^{\delta'_p}).
    \end{aligned}
    \end{equation}
    Meanwhile,
    \[
    F \cdot ((E^{1 - \delta_1} F^{\delta'_1}) \otimes \cdots \otimes (E^{1 - \delta_p} F^{\delta'_p}))
    \]
    is given by
    \begin{equation}\label{eq:FActingOnLeftTensor}
    \sum_{i: \delta'_i = 0} (-1)^{(1-\delta_1) + \delta'_1 + \cdots + (1-\delta_{i-1}) + \delta'_{i-1} + (1-\delta_i)} (E^{1 - \delta_1} F^{\delta'_1}) \otimes \cdots \otimes (E^{1 - \delta_i} F) \otimes \cdots \otimes (E^{1 - \delta_p} F^{\delta'_p}).
    \end{equation}
    Since \eqref{eq:FActingOnLeftWedgeEnd} and \eqref{eq:FActingOnLeftTensor} correspond under our map, our map respects left multiplication by $F$.

    Now we consider right multiplication by the element $E_i := 1 \otimes \cdots \otimes 1 \otimes E \otimes 1 \otimes \cdots \otimes 1$ of $(\Z\langle E,F \rangle / (E^2, F^2, EF + FE ))^{\otimes p}$. Computing
    \[
    (\varepsilon_{\Pc_{\mathrm{closed}}} \otimes (e_1^{\delta_1} \wedge \sigma_1^{\delta'_1}) \wedge \cdots \wedge (e_p^{\delta_p} \wedge \sigma_p^{\delta'_p})) \cdot E_i,
    \]
    we get zero if $\delta_i = 0$, while if $\delta_i = 1$ we get
    \begin{equation}\label{eq:EiActingOnRightWedge}
    \begin{aligned}
    &(-1)^{\delta'_i + \delta_{i+1} + \delta'_{i+1} + \cdots + \delta_p + \delta'_p} (-1)^{p-i} \\
    &\cdot \varepsilon_{\Pc_{\mathrm{closed}}} \otimes (e_1^{\delta_1} \wedge \sigma_1^{\delta'_1}) \wedge \cdots \wedge (e_i \wedge \sigma_i^{\delta'_i}) \wedge \cdots \wedge (e_p^{\delta_p} \wedge \sigma_p^{\delta'_p}).
    \end{aligned}
    \end{equation}
    Meanwhile,
    \[
    ((E^{1 - \delta_1} F^{\delta'_1}) \otimes \cdots \otimes (E^{1 - \delta_p} F^{\delta'_p})) \cdot E_i
    \]
    is also zero if $\delta_i = 0$, while if $\delta_i = 1$ we get
    \begin{equation}\label{eq:EiActingOnRightTensor}
    (-1)^{\delta'_i + (1-\delta_{i+1}) + \delta'_{i+1} + \cdots + (1-\delta_p) + \delta'_p} (E^{1 - \delta_1} F^{\delta'_1}) \otimes \cdots \otimes (E F^{\delta'_i}) \otimes \cdots \otimes (E^{1 - \delta_p} F^{\delta'_p}).
    \end{equation}
    Since \eqref{eq:EiActingOnRightWedge} and \eqref{eq:EiActingOnRightTensor} correspond under our map, our map respects right multiplication by $E_i$.

    Finally, we consider right multiplication by the element $F_i := 1 \otimes \cdots \otimes 1 \otimes F \otimes 1 \otimes \cdots \otimes 1$ of $(\Z\langle E,F \rangle / (E^2, F^2, EF + FE ))^{\otimes p}$. The action
    \[
        (\varepsilon_{\Pc_{\mathrm{closed}}} \otimes (e_1^{\delta_1} \wedge \sigma_1^{\delta'_1}) \wedge \cdots \wedge (e_p^{\delta_p} \wedge \sigma_p^{\delta'_p})) \cdot F_i
    \]
    is zero if $\delta'_i = 1$, while if $\delta'_i = 0$ we get
    \begin{equation}\label{eq:FiActingOnRightWedge}
    (-1)^{\delta_{i+1} + \delta'_{i+1} + \cdots + \delta_p + \delta'_p} (-1)^{p-i} \varepsilon_{\Pc_{\mathrm{closed}}} \otimes (e_1^{\delta_1} \wedge \sigma_1^{\delta'_1}) \wedge \cdots \wedge (e_i^{\delta_i} \wedge \sigma_i) \wedge \cdots \wedge (e_p^{\delta_p} \wedge \sigma_p^{\delta'_p}).
    \end{equation}
    Meanwhile,
    \[
    ((E^{1 - \delta_1} F^{\delta'_1}) \otimes \cdots \otimes (E^{1 - \delta_p} F^{\delta'_p})) \cdot F_i
    \]
    is also zero if $\delta'_i = 1$, while if $\delta'_i = 0$ we get
    \begin{equation}\label{eq:FiActingOnRightTensor}
    (-1)^{(1-\delta_{i+1}) + \delta'_{i+1} + \cdots + (1-\delta_p) + \delta'_p} (E^{1 - \delta_1} F^{\delta'_1}) \otimes \cdots \otimes (E^{1 - \delta_i} F) \otimes \cdots \otimes (E^{1 - \delta_p} F^{\delta'_p}).
    \end{equation}
    Since \eqref{eq:FiActingOnRightWedge} and \eqref{eq:FiActingOnRightTensor} correspond under our map, we have proved the proposition.
\end{proof}

\subsection{Projectivity of the actions} Since the TQFT $\Zb^{GY}_{\ext}$ involves categories of projective modules, we should investigate when $\Zb^{P}_{\delta,\pi}(F)$ is projective as a left module or right module.

\begin{proposition}\label{prop:Projective}
    Let $M_2 \xleftarrow{F} M_1$ be a morphism in $\Cob^{\ext}$. For any $\delta \in \Q$ and $\pi \in \Z/2\Z$:
    \begin{itemize}
        \item If each component of $F$ intersects the incoming boundary $M_1$ then the left action of $A_{EF}(M_2)$ on $\Zb^P_{\delta,\pi}(F)$ is projective. 
        \item if each component of $F$ intersects the outgoing boundary $M_2$ then the right action of $A_{EF}(M_1)$ on $\Zb^P_{\delta,\pi}(F)$ is projective.
    \end{itemize} 
\end{proposition}

\begin{proof}
    We will prove the first statement; the second is similar. For a component $F_0$ of $F$, pick a distinguished point $p \in F_0 \cap P \cap M_1$. We can arrange that all basis arcs in $F_0$ have one endpoint at $p$, and we can choose the component of $M_1$ containing $p$ as the unique boundary circle in $F_0$ to not get a basis circle around it. Using the resulting $\Z$-basis, one can see that $\Zb^P_{\delta,\pi}(F)$ is free as a left module over $A_{EF}(M_2)$ with an $A_{EF}(M_2)$-basis consisting of all wedge products $\omega$ such that:
    \begin{itemize}
        \item For each circle component $C$ of $M_2$, $\omega$ is divisible by the corresponding basis arc $e_C$ but not by the corresponding basis circle $\sigma_C$;
        \item For each interval component $I$ of $M_2$, $\omega$ is divisible by the corresponding basis arc $e_I$.
    \end{itemize}
    It follows that $\Zb^P_{\delta,\pi}(F)$ is projective as a left module over $A_{EF}(M_2)$.
\end{proof}

\begin{remark}
Over a field, the rings $A_{EF}(M)$ are local, so projective modules are free. If any component of $F$ is disjoint from $M_1$, then one can construct a nonzero element of $A_{EF}(M_2)$ that acts as zero on $\Zb^P_{\delta,\pi}(F)$, contradicting freeness, so the converse to Proposition~\ref{prop:Projective} holds over a field.
\end{remark}

\begin{remark}
    Recall that for an object $M$ of $\Cob^{ext}$, the algebra $A(M)$ considered in \cite{ManionDHASigns} is a tensor product only of copies of $\Z[E]/(E^2)$, indexed only by interval components of $M$ and not circle components. In \cite{ManionDHASigns}, projectivity of the actions of $A(M_2)$ and $A(M_1)$ on $\Zb^{S_+}_{\delta,\pi}(F)$ satisfies properties analogous to those in Proposition~\ref{prop:Projective}; if each component of $F$ intersects $M_1$ nontrivially then the left action of $A(M_2)$ is projective, and similarly for $M_2$ and the right action of $A(M_1)$. The converse holds if we work over a field.
\end{remark} 

\subsection{Relating the larger and smaller state spaces}

 We now discuss how to recover the state spaces $\Zb^{S_+}_{\delta,\pi}(F)$ of \cite{ManionDHASigns} from the spaces $\Zb^{P}_{\delta,\pi}(F)$ we consider here. Let $M_2 \xleftarrow{F} M_1$ be a morphism in $\Cob^{\ext}$; when defining $\Zb^{S_+}_{\delta,\pi}(F)$ with its actions of $A(M_i)$, it does not matter algebraically whether $S_+$ circles of $F$ are considered as incoming or outgoing, so for simplicity we will assume all $S_+$ circles of $F$ are in its incoming boundary $M_1$.

\begin{proposition}\label{prop:SPlusFromP}
    For any $\delta \in \Q$ and $\pi \in \Z/2\Z$, we have
    \[
    \Zb^{S_+}_{\delta,\F_2}(F) \cong \Zb^P_{\delta,\F_2}(F) \otimes_{(U^{\F_2}(\psl(1|1)))^{\otimes p}} (V(0,0)^{\F_2})^{\otimes p}
    \]
    and
    \[
    \Zb^{S_+}_{\delta,\pi}(F) \cong \Zb^P_{\delta,\pi}(F) \otimes_{(U^{\Z}(\psl(1|1)))^{\otimes p}} (V(0,0)_{\overline{0}})^{\otimes p}.
    \]
\end{proposition}

\begin{proof}
    The tensor products on the right are the quotients of $\Zb^P_{\delta,\F_2}(F)$ and $\Zb^P_{\delta,\pi}(F)$ by the image of the endomorphism $F_C$ for all circle components $C$ of $S_+$. Choosing bases as in Section~\ref{sec:ChoosingBases}, this quotient annihilates all basis elements divisible by the class $[C]$ of any $S_+$ circle $C$ and imposes no other relations on the remaining basis elements. The remaining basis elements are the same as the basis elements chosen for $\Zb^{S_+}_{\delta,\F_2}(F)$ and $\Zb^{S_+}_{\delta,\pi}(F)$ in \cite{ManionDHA,ManionDHASigns}.
\end{proof}

\section{A gluing theorem for intervals and circles}\label{sec:GluingTheorem}

\subsection{Interval gluing lemma}

Let $F$ be a sutured surface with all components of $S_+$ outgoing, i.e. such that $F$ corresponds to a morphism in $\Cob^{\ext}$ from $\varnothing$ to $M$ for some object $M$. Let $I_1$ and $I_2$ be two interval components of $S_+$ containing points $p_1$ and $p_2$ of $P$ and having associated maps $E_1$ and $E_2$. Let $\overline{F}$ be the result of gluing $I_1$ and $I_2$ together compatibly with the orientations.

\begin{lemma}\label{lem:IntervalGluing}[cf. Lemma 4.1 of \cite{ManionDHA}, Lemma 3.1 of \cite{ManionDHASigns}]
For any $A \in \Q$ and any $j \in \{1/2, 1\}$ such that $\pi_j(F)$ is defined, it follows that $\pi_j(\overline{F})$ is also defined and we have an isomorphism
\[
\Zb^P_{\delta_A,\pi_j}(\overline{F}) \cong \frac{\Zb^P_{\delta_A,\pi_j}}{\im(E_1 + E_2)}
\]
of $\Q$-graded super abelian groups, compatible with the left actions of $\Z[E]/(E^2)$ for interval components of $S_+$ that are neither $I_1$ nor $I_2$ and $\Z\langle E,F \rangle / (E^2, F^2, EF + FE)$ for circle components of $S_+$. Over $\F_2$ we have an analogous isomorphism with no requirement on well-definedness of $\pi_j$.
\end{lemma}

\begin{proof}
Well-definedness of $\pi_j(\overline{F})$ for $j = 1$ follows from the fact that interval gluing does not change the number of $S_+$ circles. For $j = 1/2$ it follows because while gluing $I_1$ to $I_2$ decreases the number of $S_+$ intervals by two, it also increases or decreases the number of boundary components of $F$ intersecting $S_-$ by one. 

For the isomorphism, the proof largely follows \cite[proof of Lemma 3.1]{ManionDHASigns}; we just note the modifications required to adapt that proof to the current setting. Let $\delta := \delta_A$ and $\pi := \pi_j$. 

\smallskip

\noindent \textbf{Case 1-2: only $I_2$ is alone.} When we show the isomorphism intertwines the remaining actions, there are now more types of remaining actions that could exist ($E$ and $F$ for circles along with $E$ for intervals). The argument in \cite{ManionDHASigns} for an interval $E$ action works equally well for circle $E$ actions. If $C$ is a circle component of $S_+$ with corresponding $F$-endomorphism $F_C$, then in the notation of \cite[proof of Lemma 3.1]{ManionDHASigns} we have
\[
F_C(\varepsilon_F \otimes (e \wedge \omega')) = (-1)^{\pi(F)} \varepsilon_F \otimes ([C] \wedge e \wedge \omega') = (-1)^{\pi(F) + 1} \varepsilon_F \otimes (e \wedge [C] \wedge \omega')
\]
in $\Zb^P_{\delta,\pi}(F)$. Under our isomorphism, this element gets sent to 
\[
(-1)^{\pi(F) +1 } \varepsilon_{\overline{F}} \otimes ([C] \wedge \omega')
\]
in $\Zb^P_{\delta,\pi}(\overline{F})$, which is also the result of acting with $F_C$ on $\varepsilon_{\overline{F}} \otimes ([C] \wedge \omega')$ because $\pi(\overline{F}) = \pi(F) + 1$.

\smallskip

\noindent \textbf{Case 1-3: neither $I_1$ nor $I_2$ is alone.} Again we must consider the intertwining property of our isomorphism with respect to more types of actions. Circle $E$ actions follow from the same proof as for interval $E$ actions. For a circle $C$ in $S_+$ with corresponding $F$-endomorphism $F_C$, we have
\[
F_C(\varepsilon_F \otimes (e_1 \wedge e_2 \wedge \omega')) = (-1)^{\pi(F) + 2} \varepsilon_F \otimes (e_1 \wedge e_2 \wedge [C] \wedge \omega')
\]
and
\[
F_C(\varepsilon_F \otimes (e_1\wedge \omega')) = (-1)^{\pi(F) + 1} \varepsilon_F \otimes (e_1 \wedge [C] \wedge \omega').
\]
Meanwhile, if $F_{\overline{C}}$ denotes the analogous endomorphism in the $\overline{F}$ case, we have
\[
F_{\overline{C}}(\varepsilon_{\overline{F}} \otimes (e \wedge \omega')) = (-1)^{\pi(\overline{F}) + 1} \varepsilon_{\overline{F}} \otimes (e \wedge [C] \wedge \omega')
\]
and
\[
F_{\overline{C}}(\varepsilon_{\overline{F}} \otimes \omega') = (-1)^{\pi(\overline{F})} \varepsilon_{\overline{F}} \otimes ([C] \wedge \omega').
\]
The intertwining property follows from $\pi(\overline{F}) = \pi(F) + 1$.

\smallskip

\noindent \textbf{Case 2-1a: $I_1$ and $I_2$ are not alone in their component of $F$.} If $C$ is a circle component of $S_+$ (whether or not it contains $q$), then
\[
F_C(\varepsilon_F \otimes (e_1 \wedge e_2 \wedge \omega')) = (-1)^{\pi(F) + 2} \varepsilon_F \otimes (e_1 \wedge e_2 \wedge [C] \wedge \omega')
\]
and
\[
F_C(\varepsilon_F \otimes (e_1 \wedge \omega')) = (-1)^{\pi(F) + 1} \varepsilon_F \otimes (e_1 \wedge [C] \wedge \omega').
\]
On the $\overline{F}$ side, we have
\[
F_C(\varepsilon_{\overline{F}} \otimes (e \wedge \omega')) = (-1)^{\pi(\overline{F}) + 1} \varepsilon_F \otimes (e \wedge [C] \wedge \omega')
\]
and
\[
F_C(\varepsilon_{\overline{F}} \otimes \omega') = (-1)^{\pi(\overline{F})} \varepsilon_F \otimes ([C] \wedge \omega').
\]
The intertwining property again follows from $\pi(\overline{F}) = \pi(F) + 1$. 

\smallskip

\noindent \textbf{Cases 2-1b and 2-2:} Arguments as in the above cases continue to work here.
\end{proof}

\subsection{Circle gluing lemma}

As in the previous section, let $F$ be a sutured surface with all components of $S_+$ outgoing, i.e. such that $F$ corresponds to a morphism in $\Cob^{\ext}$ from $\varnothing$ to $M$ for some object $M$. Let $C_1$ and $C_2$ be two circle components of $S_+$ containing points $p_1$ and $p_2$ of $P$ and having associated maps $E_1, F_1$ and $E_2, F_2$. Let $\overline{F}$ be the result of gluing $C_1$ and $C_2$ together compatibly with the orientations.

\begin{lemma}\label{lem:CircleGluing}
For any $A \in \Q$ and any $j \in \{1/2, 1\}$ such that $\pi_j(F)$ is defined, it follows that $\pi_j(\overline{F})$ is also defined and we have an isomorphism
\[
\Zb^P_{\delta_A,\pi_j}(\overline{F}) \cong \frac{\Zb^P_{\delta_A,\pi_j}}{\im(E_1 + E_2)+\im(F_1 + F_2)}
\]
of $\Q$-graded super abelian groups, compatible with the left actions of $\Z[E]/(E^2)$ for interval components of $S_+$ and $\Z\langle E,F \rangle / (E^2, F^2, EF + FE)$ for circle components that are neither $C_1$ nor $C_2$. Over $\F_2$ we have an analogous isomorphism with no requirement on well-definedness of $\pi_j$.
\end{lemma}

\begin{proof}

Well-definedness of $\pi_j(\overline{F})$ when $j=1$ follows because gluing along $S_+$ circles does not change the number of $S_+$ circles modulo 2. When $j = 1/2$ it follows because gluing along $S_+$ circles changes neither the number of $S_+$ intervals nor the number of boundary components of $F$ intersecting $S_-$ nontrivially.  

For the isomorphism, let $\delta := \delta_A$ and $\pi := \pi_j$. We will consider various cases.

\smallskip

\noindent \textbf{Case 1: $C_1$ and $C_2$ are on distinct components of $F$.} Choose a $\Z$-basis for $\Zb^P_{\delta,\pi}(F)$ as in Section~\ref{sec:ChoosingBases}; without loss of generality we may assume that $C_1$ and $C_2$ are each incident with at most one basis arc. 

\smallskip
\noindent \textbf{Case 1-1: $C_1$ and $C_2$ are both alone.} Assume that $C_1$ and $C_2$ are each the unique component of $S_+$ in their component of $F$.

\smallskip
\noindent \textbf{Case 1-1a: there are no $S_-$ circles in the same components of $F$ as $C_1$ or $C_2$.} In this case the gluing produces a closed component of $F$. The endomorphisms $E_1$, $E_2$, $F_1$, and $F_2$ are each zero individually, so $E_1 + E_2$ and $F_1 + F_2$ are zero. The gluing changes the quantities relevant for $\delta$ as follows:
\begin{itemize}
    \item The number of no-$S_-$ non-closed components decreases by $2$.
    \item The number of closed components increases by $1$.
    \item The number of $S_+$ circles decreases by $2$.
\end{itemize}
The resulting change in $\delta$ is
\[
-2A + (2A - 1) - 2(-1/2) = 0,
\]
so the change in $\pi$ is also zero. The statement now follows from $H_1(F,P) \cong H_1(\overline{F}, \overline{P})$; the intertwining property holds because the components of $F$ containing $C_1$ and $C_2$ intersect no other components of $S_+$. 

\smallskip
\noindent \textbf{Case 1-1b: there is an $S_-$ circle in the component with $C_1$, but not in the component with $C_2$.} The gluing changes the quantities relevant for $\delta$ as follows:
\begin{itemize}
    \item The number of no-$S_+$ non-closed components increases by $1$.
    \item The number of no-$S_-$ non-closed components decreases by $1$.
    \item The number of $S_+$ circles decreases by $2$.
\end{itemize}
The resulting change in $\delta$ is
\[
(A-1) -(A) -2(-1/2) = 0,
\]
so the change in $\pi$ is also zero.

When choosing bases, we can ensure that $[C_1]$ is a basis circle. Before gluing, basis elements of $\Zb^P_{\delta,\pi}(F)$ are either $\varepsilon_F \otimes ([C_1] \wedge \omega')$ or $\varepsilon_F \otimes \omega'$ for some wedge product $\omega'$ of basis arcs and circles not divisible by $[C_1]$.

The maps $E_1$ and $E_2$ are both zero individually, so $E_1 + E_2 = 0$. For $F_1$, we have
\begin{itemize}
\item $F_1(\varepsilon_F \otimes ([C_1] \wedge \omega')) = 0$;
\item $F_1(\varepsilon_F \otimes \omega') = (-1)^{\pi(F)} \varepsilon_F \otimes ([C_1] \wedge \omega')$.
\end{itemize}
Since $F_2$ is the zero map, we have $F_1 + F_2 = F_1$ and we are taking the quotient of $\Zb^P_{\delta,\pi}(F)$ by the image of $F_1$. A basis for the quotient is given by the elements $\varepsilon_F \otimes \omega'$ with $\omega'$ not divisible by $[C_1]$. These same elements give a basis for $\Zb^P_{\delta,\pi}(\overline{F})$, and we have a bijection sending $\varepsilon_F \otimes \omega'$ for $F$ to $\varepsilon_F \otimes \omega'$ for $\overline{F}$. Thus, we have an isomorphism of $\Q$-graded super abelian groups; the intertwining property holds because the components of $F$ containing $C_1$ and $C_2$ intersect no other components of $S_+$.

The case where there is an $S_-$ circle in the component with $C_2$ but not in the component with $C_1$ is analogous to this case, so we will not consider it separately.

\smallskip
\noindent \textbf{Case 1-1c: there is an $S_-$ circle in both the component with $C_1$ and the component with $C_2$.} The gluing changes the quantities relevant for $\delta$ as follows:
\begin{itemize}
    \item $h$ increases by $1$ (because $[C_1]$ becomes nonzero in $H_1(\overline{F},S_+)$).
    \item The number of no-$S_+$ non-closed components increases by $1$.
    \item The number of $S_+$ circles decreases by $2$.
\end{itemize}
The resulting change in $\delta$ is
\[
(-A) + (A-1) +2(-1/2) = 0,
\]
so the change in $\pi$ is also zero.

When choosing bases, we can ensure that both $[C_1]$ and $[C_2]$ are basis circles. Before gluing, some other boundary circle in the component of $F$ containing $C_1$ was not a basis circle, and similarly for the component containing $C_2$. After gluing, the glued component now has two boundary circles that were not basis circles, but instead of making one of them into a basis circle we will let $[C_1]$ be a basis circle (modulo the other basis circles $[C_1]$ is homologous to either of the non-basis boundary circles up to sign). 

Before gluing, basis elements of $\Zb^P_{\delta,\pi}(F)$ can take the following forms:
\begin{itemize}
    \item $\varepsilon_F \otimes ([C_1] \wedge [C_2] \wedge \omega')$,
    \item $\varepsilon_F \otimes ([C_1] \wedge \omega')$,
    \item $\varepsilon_F \otimes ([C_2] \wedge \omega')$,
    \item $\varepsilon_F \otimes \omega'$
\end{itemize}
where $\omega'$ is divisible by neither $[C_1]$ nor $[C_2]$. The maps $E_1$ and $E_2$ are both zero individually; we have
\begin{itemize}
    \item $F_1(\varepsilon_F \otimes ([C_1] \wedge [C_2] \wedge \omega')) = 0$,
    \item $F_1(\varepsilon_F \otimes ([C_1] \wedge \omega')) = 0$,
    \item $F_1(\varepsilon_F \otimes ([C_2] \wedge \omega')) = (-1)^{\pi(F)} \varepsilon_F \otimes ([C_1] \wedge [C_2] \wedge \omega')$,
    \item $F_1(\varepsilon_F \otimes \omega') = (-1)^{\pi(F)} \varepsilon_F \otimes ([C_1] \wedge \omega')$
\end{itemize}
and
\begin{itemize}
    \item $F_2(\varepsilon_F \otimes ([C_1] \wedge [C_2] \wedge \omega')) = 0$,
    \item $F_2(\varepsilon_F \otimes ([C_1] \wedge \omega')) = (-1)^{\pi(F) + 1} \varepsilon_F \otimes ([C_1] \wedge [C_2] \wedge \omega')$,
    \item $F_2(\varepsilon_F \otimes ([C_2] \wedge \omega')) = 0$,
    \item $F_2(\varepsilon_F \otimes \omega') = (-1)^{\pi(F)} \varepsilon_F \otimes ([C_2] \wedge \omega')$.
\end{itemize}
Thus, in the quotient by $\im(F_1 + F_2)$, we have $\varepsilon_F \otimes ([C_1] \wedge [C_2] \wedge \omega') = 0$ and
\[
\varepsilon_F \otimes ([C_2] \wedge \omega') = -\varepsilon_F \otimes ([C_1] \wedge \omega').
\]
A basis for the quotient is given by the elements $\varepsilon_F \otimes \omega'$ and $\varepsilon_F \otimes ([C_1] \wedge \omega')$. By construction, these basis elements naturally correspond to our basis elements for $\Zb^P_{\delta,\pi}(\overline{F})$, and we have a bijection sending $\varepsilon_F \otimes \omega'$ for $F$ to $\varepsilon_F \otimes \omega'$ for $\overline{F}$ and sending $\varepsilon_F \otimes ([C_1] \wedge \omega')$ for $F$ to $\varepsilon_F \otimes ([C_1] \wedge \omega')$ for $\overline{F}$. Thus, we have an isomorphism of $\Q$-graded super abelian groups; the intertwining property holds for the same reason as in the previous cases.

\smallskip
\noindent \textbf{Case 1-2: only $C_2$ is alone.}  Now assume that $C_2$ is the unique component of $S_+$ in its component of $F$ but that $C_1$ has another component of $S_+$ in its component of $F$.

\smallskip

\noindent \textbf{Case 1-2a: there are no $S_-$ circles in the same component as $C_2$.} The gluing changes the quantities relevant for $\delta$ as follows, whether or not the component of $F$ containing $C_1$ intersects $S_-$ nontrivially:
\begin{itemize}
    \item $h$ decreases by $1$.
    \item The number of no-$S_-$ non-closed components decreases by $1$.
    \item The number of $S_+$ circles decreases by $2$.
\end{itemize}
The resulting change in $\delta$ is
\[
-(-A)-(A)-2(-1/2) = +1,
\]
so $\pi$ also changes by $1$.

When choosing bases, we can ensure that $C_1$ is a basis circle and that there is a unique basis arc $e_1$ incident with $C_1$. Orient $e_1$ so that it points from the surface into $C_1$. Basis elements for $\Zb^P_{\delta,\pi}(F)$ can take the following forms:
\begin{itemize}
    \item $\varepsilon_F \otimes (e_1 \wedge [C_1] \wedge \omega')$,
    \item $\varepsilon_F \otimes ([C_1] \wedge \omega')$,
    \item $\varepsilon_F \otimes (e_1 \wedge \omega')$,
    \item $\varepsilon_F \otimes \omega'$.
\end{itemize}
We have
\begin{itemize}
    \item $E_1(\varepsilon_F \otimes (e_1 \wedge [C_1] \wedge \omega')) = (-1)^{\pi(F)}\varepsilon_F \otimes ([C_1] \wedge \omega')$,
    \item $E_1(\varepsilon_F \otimes ([C_1] \wedge \omega')) = 0$,
    \item $E_1(\varepsilon_F \otimes (e_1 \wedge \omega')) = (-1)^{\pi(F)} \varepsilon_F \otimes \omega'$,
    \item $E_1(\varepsilon_F \otimes \omega') = 0$.
\end{itemize}
The map $E_2$ is zero, so in the quotient by $\im(E_1+E_2)$ we set the basis elements of the form $\varepsilon_F \otimes \omega'$ and $\varepsilon_F \otimes ([C_1] \wedge \omega')$ to zero.

The endomorphisms $F_1$ and $F_2$ of $\Zb^{P}_{\delta,\pi}(F)$ descend to endomorphisms of the quotient by $\im(E_1 + E_2)$; we have $F_2 = 0$ and
\begin{itemize}
\item $F_1(\varepsilon_F \otimes (e_1 \wedge [C_1] \wedge \omega')) = 0$,
\item $F_1(\varepsilon_F \otimes (e_1 \wedge \omega')) = (-1)^{\pi(F) + 1} \varepsilon_F \otimes (e_1 \wedge [C_1] \wedge \omega')$.
\end{itemize}
When we take the further quotient by $\im(F_1 + F_2)$, we set basis elements of the form $\varepsilon_F \otimes (e_1 \wedge [C_1] \wedge \omega')$ to zero, and we are left with basis elements of the form $\varepsilon_F \otimes (e_1 \wedge \omega')$.

Meanwhile, in $\overline{F}$, neither $e_1$ or $[C_1]$ is a basis arc or basis circle anymore; a basis is given by elements $\varepsilon_{\overline{F}} \otimes \omega'$ where $\omega'$ runs over the same wedge products as above (not divisible by $e_1$ or $[C_1]$). We have a bijection between the bases for the quotient and for $\Zb^P_{\delta,\pi}(\overline{F})$ given by
\[
\varepsilon_F \otimes (e_1 \wedge \omega') \leftrightarrow \varepsilon_{\overline{F}} \otimes \omega'.
\]
Because $\delta$ and $\pi$ are each one more for $\overline{F}$ than they are for $F$, this bijection gives an isomorphism of $\Q$-graded super abelian groups.

To see that this bijection intertwines the remaining actions of $E$ generators, let $X \notin \{C_1,C_2\}$ be an interval or circle component of $S_+$. First assume that $e_1$ is not incident with $X$; in $\frac{\Zb^{P}_{\delta,\pi}(F)}{\im(E_1 + E_2) + \im(F_1 + F_2)}$ we have
\[
E_X(\varepsilon_F \otimes (e_1 \wedge \omega')) = (-1)^{\pi(F) + 1} \varepsilon_F \otimes (e_1 \wedge E_X(\omega')),
\]
while in $\Zb^P_{\delta,\pi}(\overline{F})$ we have
\[
E_X(\varepsilon_{\overline{F}} \otimes \omega') = (-1)^{\pi(\overline{F})} \varepsilon_{\overline{F}} \otimes E_X(\omega').
\]
These elements are identified under our bijection because $\pi(\overline{F}) = \pi(F) + 1$. If $e_1$ is incident with $X$, the extra ``remove $e_1$'' term we would get in $E_X(\varepsilon_F \otimes (e_1 \wedge \omega'))$ is zero in $\frac{\Zb^{P}_{\delta,\pi}(F)}{\im(E_1 + E_2)}$ even without the additional quotient by $\im(F_1 + F_2)$, so the analysis is unchanged.

To see that the bijection intertwines the remaining actions of $F$ generators, let $C \notin \{C_1,C_2\}$ be a circle component of $S_+$. If $C$ is not in the same component of $F$ as $C_1$, then whether or not $C$ is a basis circle, the expansion of $[C]$ in terms of basis circles does not involve $[C_1]$, and we have
\[
F_C(\varepsilon_F \otimes (e_1 \wedge \omega')) = (-1)^{\pi(F) + 1} \varepsilon_F \otimes (e_1 \wedge ([C] \wedge \omega')) \in \frac{\Zb^{P}_{\delta,\pi}(F)}{\im(E_1 + E_2) + \im(F_1 + F_2)}.
\]
We also have
\[
F_C(\varepsilon_{\overline{F}} \otimes \omega') = (-1)^{\pi(\overline{F})} \varepsilon_{\overline{F}} \otimes ([C] \wedge \omega') \in \Zb^P_{\delta,\pi}(\overline{F});
\]
these elements are identified under our bijection. If $C$ is in the same component of $F$ as $C_1$ but $C$ is also a basis circle, the argument is unchanged. If $C$ is in the same component of $F$ as $C_1$ and is not a basis circle, then taking the wedge products of our four types of basis elements for $\Zb^P_{\delta,\pi}(F)$ with $[C]$, there is an extra term that multiplies by $[C_1]$ rather than changing $\omega'$. However, this extra term vanishes in $\frac{\Zb^{P}_{\delta,\pi}(F)}{\im(F_1 + F_2)}$ and thus in $\frac{\Zb^{P}_{\delta,\pi}(F)}{\im(E_1 + E_2) + \im(F_1 + F_2)}$. In $\Zb^P_{\delta,\pi}(\overline{F})$ the basis expansion of $[C]$ is missing the $[C_1]$ term, so our bijection intertwines the actions of $F_C$.

\smallskip

\noindent \textbf{Case 1-2b: there is at least one $S_-$ circle in the same component of $F$ as $C_2$.} If the component of $F$ containing $C_1$ is disjoint from $S_-$, then the gluing changes the quantities relevant for $\delta$ as follows:
\begin{itemize}
    \item $h$ decreases by $1$.
    \item The number of no-$S_-$ non-closed components decreases by $1$.
    \item The number of $S_+$ circles decreases by $2$.
\end{itemize}
The resulting change in $\delta$ is
\[
-(-A) - (A) -2 (-1/2) = +1,
\]
so $\pi$ also changes by $1$. On the other hand, if the component of $F$ containing $C_1$ intersects $S_-$ nontrivially, then $h$ is unchanged (in terms of bases for $H_1(F,S_+)$ rather than $H_1(F,P)$ we lose a basis arc but gain a basis circle) and the quantities relevant for $\delta$ change as follows:
\begin{itemize}
    \item The number of $S_+$ circles decreases by $2$.
\end{itemize}
The resulting change in $\delta$ is
\[
-2(-1/2) = +1,
\]
so $\pi$ also changes by $1$.

When choosing bases, we can ensure that $C_1$ and $C_2$ are basis circles; we can also ensure that $C_1$ is incident with a unique basis arc $e_1$ and that $e_1$ points from the surface into $C_1$.

With respect to divisibility by $e_1$, basis elements are of the form $\varepsilon_F \otimes \omega'$ or $\varepsilon_F \otimes (e_1 \wedge \omega')$ where $\omega'$ is not divisible by $e_1$. We have
\begin{itemize}
    \item $E_1(\varepsilon_F \otimes \omega') = 0$,
    \item $E_1 (\varepsilon_F \otimes (e_1 \wedge \omega')) = (-1)^{\pi(F)} \varepsilon_F \otimes \omega'$
\end{itemize}
The map $E_2$ is zero, so in the quotient by $\im(E_1 + E_2)$ we set the basis elements of the form $\varepsilon_F \otimes \omega'$ to zero. Now, basis elements of the quotient by $\im(E_1 + E_2)$ are of one of the following forms with respect to divisibility by $[C_1]$ and $[C_2]$:
\begin{itemize}
    \item $\varepsilon_F \otimes (e_1 \wedge [C_1] \wedge [C_2] \wedge \omega')$,
    \item $\varepsilon_F \otimes (e_1 \wedge [C_1] \wedge \omega')$,
    \item $\varepsilon_F \otimes (e_1 \wedge [C_2] \wedge \omega')$,
    \item $\varepsilon_F \otimes (e_1 \wedge \omega')$
\end{itemize} 
where $\omega'$ is divisible by neither $e_1$ nor $[C_2]$. We have
\begin{itemize}
    \item $F_1(\varepsilon_F \otimes (e_1 \wedge [C_1] \wedge [C_2] \wedge \omega')) = 0$,
    \item $F_1(\varepsilon_F \otimes (e_1 \wedge [C_1] \wedge \omega')) = 0$,
    \item $F_1(\varepsilon_F \otimes (e_1 \wedge [C_2] \wedge \omega')) = (-1)^{\pi(F) + 1} \varepsilon_F \otimes (e_1 \wedge [C_1] \wedge [C_2] \wedge \omega')$,
    \item $F_1(\varepsilon_F \otimes (e_1 \wedge \omega')) = (-1)^{\pi(F) + 1} \varepsilon_F \otimes (e_1 \wedge [C_1] \wedge \omega')$
\end{itemize}
and
\begin{itemize}
    \item $F_2(\varepsilon_F \otimes (e_1 \wedge [C_1] \wedge [C_2] \wedge \omega')) = 0$,
    \item $F_2(\varepsilon_F \otimes (e_1 \wedge [C_1] \wedge \omega')) = (-1)^{\pi(F) + 2} \varepsilon_F \otimes (e_1 \wedge [C_1] \wedge [C_2] \wedge \omega')$,
    \item $F_2(\varepsilon_F \otimes (e_1 \wedge [C_2] \wedge \omega')) = 0$,
    \item $F_2(\varepsilon_F \otimes (e_1 \wedge \omega')) = (-1)^{\pi(F) + 1} \varepsilon_F \otimes (e_1 \wedge [C_2] \wedge \omega')$.
\end{itemize}
In the quotient by $\im(F_1 + F_2)$, we thus have $\varepsilon_F \otimes (e_1 \wedge [C_1] \wedge [C_2] \wedge \omega') = 0$ and 
\[
\varepsilon_F \otimes (e_1 \wedge [C_2] \wedge \omega') = - \varepsilon_F \otimes (e_1 \wedge [C_1] \wedge \omega').
\]
Basis elements of the quotient are of the form $\varepsilon_F \otimes (e_1 \wedge \omega')$ or $\varepsilon_F \otimes (e_1 \wedge [C_1] \wedge \omega')$.

Meanwhile, for $\overline{F}$, $e_1$ is no longer a basis arc, and rather than adding a new basis circle elsewhere around one of the remaining boundary components of $\overline{F}$, we can take $[C_1]$ as a basis circle for $\overline{F}$. Define a bijection of basis elements sending
\[
\varepsilon_F \otimes (e_1 \wedge \omega') \leftrightarrow \varepsilon_{\overline{F}} \otimes \omega'
\]
and
\[
\varepsilon_F \otimes (e_1 \wedge [C_1] \wedge \omega') \leftrightarrow \varepsilon_{\overline{F}} \otimes ([C_1] \wedge \omega').
\]
Because $\delta$ and $\pi$ are each one more for $\overline{F}$ than they are for $F$, this bijection gives an isomorphism of $\Q$-graded super abelian groups.

The intertwining property for the remaining actions of $E$ follows as above. For the remaining actions of $F$, the only potentially problematic case is if $C \notin \{C_1,C_2\}$ is in the same component of $F$ as $C_1$ but is not a basis circle. The expansion of $[C]$ in terms of basis circles has a term $\pm [C_1]$, so
\begin{equation}\label{eq:Case1-2bIntertwining}
F_C(\varepsilon_F \otimes (e_1 \wedge \omega')) = (-1)^{\pi(F)+1 } \varepsilon_F \otimes (e_1 \wedge (\pm [C_1]) \wedge \omega') + \cdots
\end{equation}
where the remaining terms are not divisible by $[C_1]$. For $\overline{F}$, because we included $[C_1]$ as a basis circle, it follows that the expansion of $[C]$ in terms of basis circles is the same as it was for $F$. We get
\[
F_C(\varepsilon_{\overline{F}} \otimes \omega') = (-1)^{\pi(\overline{F})} \varepsilon_{\overline{F}} \otimes ((\pm [C_1]) \wedge \omega') + \cdots
\]
where the remaining terms are the same as in \eqref{eq:Case1-2bIntertwining} with $e_1$ removed and with $\varepsilon_F$ replaced by $\varepsilon_{\overline{F}}$, and the two instances of $\pm$ represent the same sign. Because $\pi(\overline{F}) = \pi(F) + 1$, the intertwining property holds.

\smallskip
\noindent \textbf{Case 1-3: neither $C_1$ nor $C_2$ is alone.} If either the component of $F$ containing $C_1$ or the component of $F$ containing $C_2$ is disjoint from $S_-$, then the gluing changes the quantities relevant for $\delta$ as follows:
\begin{itemize}
    \item $h$ decreases by $1$.
    \item The number of no-$S_-$ non-closed components decreases by $1$.
    \item The number of $S_+$ circles decreases by $2$.
\end{itemize}
The resulting change in $\delta$ is $+1$ as above, so $\pi$ also changes by $1$. On the other hand, if both the component of $F$ containing $C_1$ and the component of $F$ containing $C_2$ intersect $S_-$ nontrivially, then $h$ is unchanged because we lose a basis arc and gain a basis circle; the only change in quantities relevant for $\delta$ is that the number of $S_+$ circles decreases by $2$. Thus, $\delta$ changes by $+1$ and $\pi$ changes by $1$.

When choosing bases, we can ensure that $C_1$ and $C_2$ are basis circles, that $C_i$ is incident with a unique arc $e_i$ for $i \in \{1,2\}$, and that $e_1$ points from the surface into $C_1$ while $e_2$ points from $C_2$ into the surface. With respect to divisibility by $e_1$ and $e_2$, there are four types of basis elements for $\Zb^+_{\delta,\pi}(F)$, on which $E_1$ and $E_2$ act as follows:
\begin{itemize}
\item $E_1(\varepsilon_F \otimes \omega') = 0$,
\item $E_1(\varepsilon_F \otimes (e_1 \wedge \omega')) = (-1)^{\pi(F)} \varepsilon_F \otimes \omega'$,
\item $E_1(\varepsilon_F \otimes (e_2 \wedge \omega')) = 0$,
\item $E_1(\varepsilon_F \otimes (e_1 \wedge e_2 \wedge \omega')) = (-1)^{\pi(F)} e_2 \wedge \omega'$
\end{itemize}
and
\begin{itemize}
\item $E_2(\varepsilon_F \otimes \omega') = 0$,
\item $E_2(\varepsilon_F \otimes (e_1 \wedge \omega')) = 0$,
\item $E_2(\varepsilon_F \otimes (e_2 \wedge \omega')) = (-1)^{\pi(F)+1} \varepsilon_F \otimes \omega'$,
\item $E_2(\varepsilon_F \otimes (e_1 \wedge e_2 \wedge \omega')) = (-1)^{\pi(F)+2} e_1 \wedge \omega'$.
\end{itemize}
In the quotient by $\im(E_1 + E_2)$, we see that $\varepsilon_F \otimes \omega' = 0$ and
\[
\varepsilon_F \otimes (e_2 \wedge \omega') = - \varepsilon_F \otimes (e_1 \wedge \omega').
\]
Basis elements for the quotient are of the form $\varepsilon_F \otimes (e_1 \wedge \omega')$ or $\varepsilon_F \otimes (e_1 \wedge e_2 \wedge \omega')$. With respect to divisibility by $[C_1]$ and $[C_2]$, there are four types of these elements, on which $F_1$ and $F_2$ act as follows:
\begin{itemize}
    \item $F_1(\varepsilon_F \otimes \omega') = (-1)^{\pi(F)} \varepsilon_F \otimes ([C_1] \wedge \omega')$,
    \item $F_1(\varepsilon_F \otimes ([C_1] \wedge \omega')) = 0$,
    \item $F_1(\varepsilon_F \otimes ([C_2] \wedge \omega')) = (-1)^{\pi(F)} \varepsilon_F \otimes ([C_1] \wedge [C_2] \wedge \omega')$,
    \item $F_1(\varepsilon_F \otimes ([C_1] \wedge [C_2] \wedge \omega')) = 0$
\end{itemize}
and
\begin{itemize}
    \item $F_2(\varepsilon_F \otimes \omega') = (-1)^{\pi(F)} \varepsilon_F \otimes ([C_2] \wedge \omega')$,
    \item $F_2(\varepsilon_F \otimes ([C_1] \wedge \omega')) = (-1)^{\pi(F)+1} \varepsilon_F \otimes ([C_1] \wedge [C_2] \wedge \omega')$,
    \item $F_2(\varepsilon_F \otimes ([C_2] \wedge \omega')) = 0$,
    \item $F_2(\varepsilon_F \otimes ([C_1] \wedge [C_2] \wedge \omega')) = 0$.
\end{itemize}
Thus, basis elements of the quotient by $\im(E_1 + E_2) + \im(F_1 + F_2)$ have one of the following forms:
\begin{itemize}
    \item $\varepsilon_F \otimes (e_1 \wedge \omega')$,
    \item $\varepsilon_F \otimes (e_1 \wedge e_2 \wedge \omega')$,
    \item $\varepsilon_F \otimes (e_1 \wedge [C_1] \wedge \omega')$,
    \item $\varepsilon_F \otimes (e_1 \wedge e_2 \wedge [C_1] \wedge \omega')$.
\end{itemize}

Meanwhile, in $\overline{F}$ we can take $[C_1]$ as a basis circle as usual, and the concatenation of $e_1$ and $e_2$ produces a new basis arc $e$. Make the identifications
\begin{itemize}
    \item $\varepsilon_F \otimes (e_1 \wedge \omega') \leftrightarrow \varepsilon_{\overline{F}} \otimes \omega'$,
    \item $\varepsilon_F \otimes (e_1 \wedge e_2 \wedge \omega') \leftrightarrow \varepsilon_{\overline{F}} \otimes (e \wedge \omega')$,
    \item $\varepsilon_F \otimes (e_1 \wedge [C_1] \wedge \omega') \leftrightarrow \varepsilon_{\overline{F}} \otimes ([C_1] \wedge \omega')$,
    \item $\varepsilon_F \otimes (e_1 \wedge e_2 \wedge [C_1] \wedge \omega') \leftrightarrow \varepsilon_{\overline{F}} \otimes (e \wedge [C_1] \wedge \omega')$.
\end{itemize}
These identifications give an isomorphism of $\Q$-graded super abelian groups. For the intertwining property, actions of $E$ corresponding to intervals or circles $X \notin \{C_1,C_2\}$ that intersect neither $e_1$ nor $e_2$ are dealt with as above. If $X$ contains an endpoint of $e_1$, then we have
\begin{itemize}
    \item $E_X(\varepsilon_F \otimes (e_1 \wedge \omega')) = 0$,
    \item $E_X(\varepsilon_F \otimes (e_1 \wedge e_2 \wedge \omega')) = (-1)^{\pi(F) + 2} \varepsilon_F \otimes (e_1 \wedge \omega')$ (one extra minus sign comes from removing $-e_1$ and the other comes from replacing $e_2$ with $-e_1$),
    \item $E_X(\varepsilon_F \otimes (e_1 \wedge [C_1] \wedge \omega')) = 0$,
    \item $E_X(\varepsilon_F \otimes (e_1 \wedge e_2 \wedge [C_1] \wedge \omega')) = (-1)^{\pi(F) + 2} \varepsilon_F \otimes (e_1 \wedge [C_1] \wedge \omega')$.
\end{itemize}
In $\overline{F}$, $X$ now contains an endpoint of $e$ (which points from $X$ into the surface), so 
\begin{itemize}
    \item $E_X(\varepsilon_{\overline{F}} \otimes \omega') = 0$,
    \item $E_X(\varepsilon_{\overline{F}} \otimes (e \wedge \omega')) = (-1)^{\pi({\overline{F}}) + 1} \varepsilon_{\overline{F}} \otimes \omega'$,
    \item $E_X(\varepsilon_{\overline{F}} \otimes ([C_1] \wedge \omega')) = 0$,
    \item $E_X(\varepsilon_{\overline{F}} \otimes (e \wedge [C_1] \wedge \omega')) = (-1)^{\pi(\overline{F}) + 1} \varepsilon_{\overline{F}} \otimes ([C_1] \wedge \omega')$.
\end{itemize}
Because $\pi(\overline{F}) = \pi(F) + 1$, our identifications intertwine the remaining actions of $E$.

For the remaining actions of $F$, the only case not entirely analogous to the ones above is if $C \notin \{C_1,C_2\}$ is in the same component as $C_2$ but is not a basis circle. The expansion of $[C]$ in terms of basis circles has a term $\pm [C_2]$, so
\[
F_C(\varepsilon_F \otimes (e_1 \wedge \omega')) = (-1)^{\pi(F) + 1} \varepsilon_F \otimes (e_1 \wedge (\mp [C_1]) \wedge \omega') + \cdots
\]
and
\[
F_C(\varepsilon_F \otimes (e_1 \wedge e_2 \wedge \omega')) = (-1)^{\pi(F) + 2} \varepsilon_F \otimes (e_1 \wedge e_2 \wedge (\mp [C_1]) \wedge \omega') + \cdots.
\]
The key observation is that in $\overline{F}$, the circles $C_1$ and $C_2$ are oriented oppositely, so that $[C_2] = -[C_1]$ and the basis expansion of $[C]$ in terms of basis circles has a term $\mp [C_1]$. We have
\[
F_C(\varepsilon_{\overline{F}} \otimes \omega') = (-1)^{\pi(\overline{F})} \varepsilon_{\overline{F}} \otimes ((\mp [C_1]) \wedge \omega') + \cdots
\]
and
\[
F_C(\varepsilon_{\overline{F}} \otimes (e \wedge \omega') = (-1)^{\pi(\overline{F})} \varepsilon_{\overline{F}+1} \otimes (e \wedge (\mp [C_1]) \wedge \omega') + \cdots.
\]
Thus, the formulas for $F_C$ in the $F$ and $\overline{F}$ cases agree under our identification.

\smallskip

\noindent \textbf{Case 2: $C_1$ and $C_2$ are on the same component of $F$.}

\smallskip

\noindent \textbf{Case 2-1: gluing $C_1$ and $C_2$ produces a closed component of $F$.} Assume that the component of $F$ containing $C_1$ and $C_2$ is otherwise disjoint from both $S_+$ and $S_-$, so that gluing $C_1$ and $C_2$ produces a closed component of $F$. The quantities relevant for $\delta$ change as follows:
\begin{itemize}
    \item $h$ increases by $1$ (in terms of bases for $H_1(F,S_+)$, a basis arc turned into a basis circle and we added another basis circle).
    \item The number of no-$S_-$ non-closed components decreases by $1$.
    \item The number of closed components increases by $1$.
    \item The number of $S_+$ circles decreases by $2$.
\end{itemize}
The resulting change in $\delta$ is
\[
(-A) - (A) + (2A - 1) -2 (-1/2) = 0
\]
so $\pi$ is also unchanged.

When choosing bases, we can ensure that $C_1$ is a basis circle and that the only basis arc intersecting $C_1$ or $C_2$ is an arc $e$ pointing out of $C_2$ and into $C_1$. We have
\begin{itemize}
    \item $E_1(\varepsilon_F \otimes (e \wedge [C_1] \wedge \omega')) = (-1)^{\pi(F)} \varepsilon_F \otimes ([C_1] \wedge \omega')$,
    \item $E_1(\varepsilon_F \otimes (e \wedge \omega')) = (-1)^{\pi(F)} \varepsilon_F \otimes \omega'$,
    \item $E_1(\varepsilon_F \otimes ([C_1] \wedge \omega')) = 0$,
    \item $E_1(\varepsilon_F \otimes \omega') = 0$ 
\end{itemize}
and
\begin{itemize}
    \item $E_2(\varepsilon_F \otimes (e \wedge [C_1] \wedge \omega')) = (-1)^{\pi(F)+1} \varepsilon_F \otimes ([C_1] \wedge \omega')$,
    \item $E_2(\varepsilon_F \otimes (e \wedge \omega')) = (-1)^{\pi(F)+1} \varepsilon_F \otimes \omega'$,
    \item $E_2(\varepsilon_F \otimes ([C_1] \wedge \omega')) = 0$,
    \item $E_2(\varepsilon_F \otimes \omega') = 0$.
\end{itemize}
Thus, $E_1 + E_2 = 0$. For $i \in \{1,2\}$ we also have
\begin{itemize}
    \item $F_1(\varepsilon_F \otimes (e \wedge [C_1] \wedge \omega')) = 0$,
    \item $F_1(\varepsilon_F \otimes (e \wedge \omega')) = (-1)^{\pi(F)+1} \varepsilon_F \otimes (e \wedge [C_1] \wedge \omega')$,
    \item $F_1(\varepsilon_F \otimes ([C_1] \wedge \omega')) = 0$,
    \item $F_1(\varepsilon_F \otimes \omega') = (-1)^{\pi(F)} \varepsilon_F \otimes ([C_1] \wedge \omega')$ 
\end{itemize}
and 
\begin{itemize}
    \item $F_2(\varepsilon_F \otimes (e \wedge [C_1] \wedge \omega')) = 0$,
    \item $F_2(\varepsilon_F \otimes (e \wedge \omega')) = (-1)^{\pi(F)+2} \varepsilon_F \otimes (e \wedge [C_1] \wedge \omega')$,
    \item $F_2(\varepsilon_F \otimes ([C_1] \wedge \omega')) = 0$,
    \item $F_2(\varepsilon_F \otimes \omega') = (-1)^{\pi(F)+1} \varepsilon_F \otimes ([C_1] \wedge \omega')$ 
\end{itemize}
so $F_1 + F_2$ is also the zero map. 

Meanwhile, in $\overline{F}$, the arc $e$ closes up to become a basis circle $\tau$ and we can retain $[C_1]$ as a basis circle. Make the identifications
\begin{itemize}
    \item $\varepsilon_F \otimes (e \wedge [C_1] \wedge \omega') \leftrightarrow \varepsilon_{\overline{F}} \otimes (\tau \wedge [C_1] \wedge \omega')$,
    \item $\varepsilon_F \otimes (e \wedge \omega') \leftrightarrow \varepsilon_{\overline{F}} \otimes (\tau  \wedge \omega')$,
    \item $\varepsilon_F \otimes ([C_1] \wedge \omega') \leftrightarrow \varepsilon_{\overline{F}} \otimes ([C_1] \wedge \omega')$,
    \item $\varepsilon_F \otimes \omega' \leftrightarrow \varepsilon_{\overline{F}} \otimes \omega'$.
\end{itemize}
These identifications give an isomorphism of $\Q$-graded super abelian groups, and there are no complications with the intertwining property.

\smallskip

\noindent \textbf{Case 2-2: the component of $F$ containing $C_1$ and $C_2$ contains no other components of $S_+$ but has at least one $S_-$ circle.} The quantities relevant for $\delta$ change as follows:
\begin{itemize}
    \item $h$ increases by $1$ (in terms of bases for $H_1(F,S_+)$, a basis arc turned into a basis circle and we added another basis circle).
    \item The number of no-$S_+$ non-closed components increases by $1$.
    \item The number of $S_+$ circles decreases by $2$.
\end{itemize}
The resulting change in $\delta$ is
\[
(-A)+ (A-1) -2 (-1/2) = 0,
\]
so $\pi$ is also unchanged.

When choosing bases, we can ensure that both $C_1$ and $C_2$ are basis circles and that the only basis arc intersecting $C_1$ or $C_2$ is an arc $e$ pointing out of $C_2$ and into $C_1$. When computing $E_1$ and $E_2$, we do not care about divisibility by $[C_i]$; we have
\begin{itemize}
    \item $E_1(\varepsilon_F \otimes (e \wedge \omega')) = (-1)^{\pi(F)} \varepsilon_F \otimes \omega'$,
    \item $E_1(\varepsilon_F \otimes \omega') = 0$
\end{itemize}
and
\begin{itemize}
    \item $E_2(\varepsilon_F \otimes (e \wedge \omega')) = (-1)^{\pi(F)+1} \varepsilon_F \otimes \omega'$,
    \item $E_2(\varepsilon_F \otimes \omega') = 0$
\end{itemize}
Thus, $E_1 + E_2$ is the zero map.

Similarly, when computing $F_1$ and $F_2$, we do not care about divisibility by $e$; we have
\begin{itemize}
    \item $F_1(\varepsilon_F \otimes ([C_1] \wedge [C_2] \wedge \omega')) = 0$,
    \item $F_1(\varepsilon_F \otimes ([C_1] \wedge \omega')) = 0$,
    \item $F_1(\varepsilon_F \otimes ([C_2] \wedge \omega') )= (-1)^{\pi(F)} \varepsilon_F \otimes ([C_1] \wedge [C_2] \wedge \omega')$,
    \item $F_1(\varepsilon_F \otimes \omega') =(-1)^{\pi(F)} \varepsilon_F \otimes ([C_1] \wedge \omega')$
\end{itemize}
and
\begin{itemize}
    \item $F_2(\varepsilon_F \otimes ([C_1] \wedge [C_2] \wedge \omega')) = 0$,
    \item $F_2(\varepsilon_F \otimes ([C_1] \wedge \omega')) = (-1)^{\pi(F)+1} \varepsilon_F \otimes ([C_1] \wedge [C_2] \wedge \omega')$,
    \item $F_2(\varepsilon_F \otimes ([C_2] \wedge \omega') )= 0$,
    \item $F_2(\varepsilon_F \otimes \omega') =(-1)^{\pi(F)} \varepsilon_F \otimes ([C_2] \wedge \omega')$.
\end{itemize}
Thus, in the quotient by $\im(F_1 + F_2)$, we have $\varepsilon_F \otimes ([C_1] \wedge [C_2] \wedge \omega') = 0$ and
\[
\varepsilon_F \otimes ([C_2] \wedge \omega') = -\varepsilon_F \otimes ([C_1] \wedge \omega').
\]
Basis elements for the quotient are of the form $\varepsilon_F \otimes \omega'$, $\varepsilon_F \otimes ([C_1] \wedge \omega')$, $\varepsilon_F \otimes (e \wedge \omega')$, and $\varepsilon(F) \otimes (e \wedge [C_1] \wedge \omega')$.

Meanwhile, in $\overline{F}$, the basis arc $e$ closes up to become a basis circle $\tau$, and we can retain $[C_1] = -[C_2]$ as a basis circle. Make the identifications
\begin{itemize}
\item $\varepsilon_F \otimes \omega' \leftrightarrow \varepsilon_F \otimes \omega'$,
\item $\varepsilon_F \otimes ([C_1] \wedge \omega') \leftrightarrow \varepsilon_F \otimes ([C_1] \wedge \omega')$,
\item $\varepsilon_F \otimes (e \wedge \omega') \leftrightarrow \varepsilon_F \otimes (\tau \wedge \omega')$,
\item $\varepsilon_F \otimes (e \wedge [C_1] \wedge \omega') \leftrightarrow \varepsilon_F \otimes (\tau \wedge [C_1] \wedge \omega')$;
\end{itemize}
we get an isomorphism of $\Q$-graded super abelian groups and again there are no complications with the intertwining property.

\smallskip

\noindent \textbf{Case 2-3: the component of $F$ containing $C_1$ and $C_2$ contains at least one other component of $S_+$.} The quantities relevant for $\delta$ change as follows: 
\begin{itemize}
    \item The number of $S_+$ circles decreases by $2$.
\end{itemize}
Note that $h$ is unchanged because, in terms of bases for $H_1(F,S_+)$, two basis arcs combined to form a basis circle ($-1$ to $h$) but we also added another basis circle ($+1$ to $h$). The resulting change in $\delta$ is $+1$, so $\pi$ also changes by $1$.

When choosing bases, we can ensure that both $C_1$ and $C_2$ are basis circles, that $C_i$ is incident with a unique basis arc $e_i$ for $i \in \{1,2\}$, and that $e_1$ points from the surface into $C_1$ while $e_2$ points from $C_2$ into the surface.

When computing $E_1$ and $E_2$, we do not care about divisibility by $[C_1]$ or $[C_2]$; we have
\begin{itemize}
    \item $E_1(\varepsilon_F \otimes (e_1 \wedge e_2 \wedge \omega')) = (-1)^{\pi(F)} \varepsilon_F \otimes (e_2 \wedge \omega')$,
    \item $E_1(\varepsilon_F \otimes (e_1 \wedge \omega')) = (-1)^{\pi(F)} \varepsilon_F \otimes \omega'$,
    \item $E_1(\varepsilon_F \otimes (e_2 \wedge \omega')) = 0$,
    \item $E_1(\varepsilon_F \otimes \omega') = 0$
\end{itemize}
and
\begin{itemize}
    \item $E_2(\varepsilon_F \otimes (e_1 \wedge e_2 \wedge \omega')) = (-1)^{\pi(F)+2} \varepsilon_F \otimes (e_1 \wedge \omega')$,
    \item $E_2(\varepsilon_F \otimes (e_1 \wedge \omega')) = 0$,
    \item $E_2(\varepsilon_F \otimes (e_2 \wedge \omega')) = (-1)^{\pi(F)+1} \varepsilon_F \otimes \omega'$,
    \item $E_2(\varepsilon_F \otimes \omega') = 0$.
\end{itemize}
In the quotient by $\im(E_1 + E_2)$, we thus have $\varepsilon_F \otimes \omega' = 0$ and
\[
\varepsilon_F \otimes (e_2 \wedge \omega') = -\varepsilon_F \otimes (e_1 \wedge \omega').
\]
Basis elements for the quotient are of the form $\varepsilon_F \otimes (e_1 \wedge e_2 \wedge \omega')$ and $\varepsilon_F \otimes (e_1 \wedge \omega')$.

When computing $F_1$ and $F_2$, we do not care about divisibility by $e_1$ or $e_2$; we have
\begin{itemize}
    \item $F_1(\varepsilon_F \otimes ([C_1] \wedge [C_2] \wedge \omega')) = 0$,
    \item $F_1(\varepsilon_F \otimes ([C_1] \wedge \omega')) = 0$,
    \item $F_1(\varepsilon_F \otimes ([C_2] \wedge \omega')) = (-1)^{\pi(F)} \varepsilon_F \otimes ([C_1] \wedge [C_2] \wedge \omega')$,
    \item $F_1(\varepsilon_F \otimes \omega') = (-1)^{\pi(F)} \varepsilon_F \otimes ([C_1] \wedge \omega')$
\end{itemize}
and
\begin{itemize}
    \item $F_2(\varepsilon_F \otimes ([C_1] \wedge [C_2] \wedge \omega')) = 0$,
    \item $F_2(\varepsilon_F \otimes ([C_1] \wedge \omega')) = (-1)^{\pi(F)+1} \varepsilon_F \otimes ([C_1] \wedge [C_2] \wedge \omega')$,
    \item $F_2(\varepsilon_F \otimes ([C_2] \wedge \omega')) = 0$,
    \item $F_2(\varepsilon_F \otimes \omega') = (-1)^{\pi(F)} \varepsilon_F \otimes ([C_2] \wedge \omega')$.
\end{itemize}
In the quotient by $\im(F_1 + F_2)$ we have $\varepsilon_F \otimes ([C_1] \wedge [C_2] \wedge \omega') = 0$ and
\[
\varepsilon_F \otimes ([C_2] \wedge \omega') = -\varepsilon_F \otimes ([C_1] \wedge \omega').
\]
Thus, we have four types of basis element for the quotient by $\im(E_1 + E_2) + \im(F_1 + F_2)$:
\begin{itemize}
    \item $\varepsilon_F \otimes (e_1 \wedge e_2 \wedge \omega')$,
    \item $\varepsilon_F \otimes (e_1 \wedge e_2 \wedge [C_1] \wedge \omega')$,
    \item $\varepsilon_F \otimes (e_1 \wedge \omega')$,
    \item $\varepsilon_F \otimes (e_1 \wedge [C_1] \wedge \omega')$.
\end{itemize}

Meanwhile, for $\overline{F}$, the arcs $e_1$ and $e_2$ combine into a basis circle $\tau$ and we can also retain $[C_1]$ as a basis circle. Make the identifications
\begin{itemize}
    \item $\varepsilon_F \otimes (e_1 \wedge e_2 \wedge \omega') \leftrightarrow \varepsilon_{\overline{F}} \otimes (\tau \wedge \omega')$,
    \item $\varepsilon_F \otimes (e_1 \wedge e_2 \wedge [C_1] \wedge \omega') \leftrightarrow \varepsilon_{\overline{F}} \otimes (\tau \wedge [C_1] \wedge \omega')$,
    \item $\varepsilon_F \otimes (e_1 \wedge \omega') \leftrightarrow \varepsilon_{\overline{F}} \otimes \omega'$,
    \item $\varepsilon_F \otimes (e_1 \wedge [C_1] \wedge \omega') \leftrightarrow \varepsilon_{\overline{F}} \otimes ([C_1] \wedge \omega')$;
\end{itemize}
we get an isomorphism of $\Q$-graded super abelian groups.

For the intertwining property, we need to see what happens if $C \notin \{C_1,C_2\}$ is not a basis circle but lives in the same component of $F$ as $C_1$ and $C_2$. In this case, the expansion of $[C]$ in terms of basis circles includes terms $\pm([C_1] + [C_2])$ and the potentially problematic terms for the intertwining property are:
\begin{itemize}
\item $F_C(\varepsilon_F \otimes (e_1 \wedge e_2 \wedge \omega')) = (-1)^{\pi(F)+2} \varepsilon_F \otimes (e_1 \wedge e_2 \wedge \pm([C_1] + [C_2]) \wedge \omega' + \cdots$,
\item $F_C(\varepsilon_F \otimes (e_1 \wedge e_2 \wedge [C_1] \wedge \omega')) = (-1)^{\pi(F) + 2} \varepsilon_F \otimes (e_1 \wedge e_2 \wedge (\pm [C_2]) \wedge [C_1] \wedge \omega') + \cdots$,
\item $F_C(\varepsilon_F \otimes (e_1 \wedge \omega')) = (-1)^{\pi(F)+1} \varepsilon_F \otimes (e_1  \wedge \pm([C_1] + [C_2]) \wedge \omega' + \cdots$,
\item $F_C(\varepsilon_F \otimes (e_1 \wedge [C_1] \wedge \omega')) = (-1)^{\pi(F) + 1} \varepsilon_F \otimes (e_1 \wedge (\pm [C_2]) \wedge [C_1] \wedge \omega') + \cdots$.
\end{itemize}
However, in the quotient $\frac{\Zb^P_{\delta,\pi}(F)}{\im(E_1 + E_2) + \im(F_1 + F_2)}$, these terms are all zero. Meanwhile, in $\overline{F}$, $[C_2]$ becomes equal to $-[C_1]$ and thus $\pm([C_1] + [C_2])$ vanishes from the basis expansion of $[C]$, so there are no analogues of the above terms when applying $F_C$ on $\Zb^P_{\delta,\pi}(\overline{F})$. The intertwining property follows.
\end{proof}

\subsection{Composing open-closed cobordisms}

While Theorem~\ref{thm:IntroMain1F2} is phrased only over $\F_2$, we will prove the following more general version.

\begin{theorem}\label{thm:GeneralCompositionWithSigns}
Let $M_1$, $M_2$, and $M_3$ be objects of $\Cob^{\ext}$ and let
\[
M_3 \xleftarrow{F'} M_2 \xleftarrow{F} M_1
\]
be morphisms in $\Cob^{\ext}$. For any $A \in \Q$, we have
\[
\Zb^P_{\delta_A,\F_2}(F' \circ F) \cong \Zb^P_{\delta_A,\F_2}(F') \otimes_{A^{\F_2}_{EF}(M_2)} \Zb^P_{\delta_A,\F_2}(F)
\]
as $\Q$-graded bimodules over the $\Z$-graded $\F_2$-algebras $(A^{\F_2}_{EF}(M_3),A^{\F_2}_{EF}(M_1))$. Furthermore, if $\pi_j(F)$ and $\pi_j(F')$ both make sense for some $j \in \{1/2, 1\}$, then $\pi_j(F' \circ F)$ also makes sense and we have
\[
\Zb^P_{\delta_A,\pi_j}(F' \circ F) \cong \Zb^P_{\delta_A,\pi_j}(F') \otimes_{A_{EF}(M_2)} \Zb^P_{\delta_A,\pi_j}(F)
\]
as $\Q$-graded bimodules over the $\Z$-graded super rings $(A_{EF}(M_3),A_{EF}(M_1))$.
\end{theorem}

\begin{proof}
The claim about $\pi_j(F' \circ F)$ follows as in the beginning of the proof of Lemma~\ref{lem:CircleGluing}. Below we will assume that either we have $j \in \{1/2, 1\}$ such that $\pi_j(F)$ and $\pi_j(F')$ make sense, or that we are working over $\F_2$; we will use the notation of the $\Z$-version. Let $\delta = \delta_A$ and $\pi = \pi_j$.

As in \cite[proof of Theorem 1.2]{ManionDHASigns}, we can write $\Zb^{P}_{\delta,\pi}(F') \otimes_{A_{EF}(M_2)} \Zb^{P}_{\delta,\pi}(F)$ as
\begin{equation}\label{eq:CompositionTensor}
\frac{\Zb^{P}_{\delta,\pi}(F') \otimes_{\Z} \Zb^{P}_{\delta,\pi}(F)}{\spann_{\Z}\{(\varepsilon_{F'} \otimes x)a \otimes (\varepsilon_{F} \otimes y) - (\varepsilon_{F'} \otimes x) \otimes a(\varepsilon_{F} \otimes y)\}}
\end{equation}
where, in the denominator, $a$ is an arbitrary multiplicative generator 
\[
a = 1 \otimes \cdots \otimes 1 \otimes E \otimes 1 \otimes \cdots \otimes 1
\]
or
\[
a = 1 \otimes \cdots \otimes 1 \otimes F \otimes 1 \otimes \cdots \otimes 1
\]
of $A_{EF}(M_2)$, $x$ is an arbitrary basis element of $\wedge^* H_1(F',P)$, and $y$ is an arbitrary basis element of $\wedge^* H_1(F,P)$.

Let $\mathrm{Left}_{A_{EF}(M_1)} \left[ \Zb^{P}_{\delta,\pi}(F') \otimes_{\Z} \Zb^{P}_{\delta,\pi}(F) \right]$ denote $\Zb^{P}_{\delta,\pi}(F') \otimes_{\Z} \Zb^{P}_{\delta,\pi}(F)$ with the right action of the supercommutative superalgebra $A_{EF}(M_1)$ viewed as a left action; an element $a$ acts on the left by first passing to the right, which picks up a sign, and then acting from the right. If we let $(F' \sqcup F)_{\mathrm{left}}$ denote $F' \sqcup F$ with all of its $S_+$ boundary components viewed as outgoing, then as $\Q$-graded super abelian groups we have
\[
\mathrm{Left}_{A_{EF}(M_1)} \left[ \Zb^{P}_{\delta,\pi}(F') \otimes_{\Z} \Zb^{P}_{\delta,\pi}(F) \right] \cong \Zb^P_{\delta,\pi}((F' \sqcup F)_{\mathrm{left}})
\]
via the map $\Phi$ sending
\[
\varepsilon_{F'} \otimes x \otimes \varepsilon_F \otimes y \mapsto (-1)^{|x|\pi(F)} \varepsilon_{(F' \sqcup F)_{\mathrm{left}}} \otimes (x \wedge y).
\]
The isomorphism $\Phi$ is compatible with the left actions of $A_{EF}(M_3)$; it is also compatible with the left actions of $F$ for incoming circles of $F$, and it relates the left actions of $E$ for incoming intervals or circles of $F$ by a minus sign (the explanation of this minus sign is the same as in \cite[proof of Theorem 1.2]{ManionDHASigns}).

The map $\Phi$ sends $(\varepsilon_{F'} \otimes x)a \otimes (\varepsilon_{F} \otimes y)$ to 
\[
(-1)^{(|x|+1)\pi(F)} \varepsilon_{(F' \sqcup F)_{\mathrm{left}}} \otimes (xa \wedge y)
\]
where $xa$ is still computed in $\Zb^P_{\delta,\pi}(F')$. In terms of the left action $\bullet_1$ (coming from $F'$) of $A(M_2)$ on $\Zb^P_{\delta,\pi}((F' \sqcup F)_{\mathrm{left}})$, we can write this element as either
\begin{align*}
&(-1)^{(|x|+1)\pi(F) + \pi(F') + \pi(F) + |x|} a \bullet_1 \varepsilon_{(F' \sqcup F)_{\mathrm{left}}} \otimes (x \wedge y) \\
&= (-1)^{|x|\pi(F) + \pi(F') + |x|} a \bullet_1 \varepsilon_{(F' \sqcup F)_{\mathrm{left}}} \otimes (x \wedge y)
\end{align*}
if $a$ is an $F$ generator or
\begin{align*}
&-(-1)^{(|x|+1)\pi(F) + \pi(F') + \pi(F) + |x|} a \bullet_1 \varepsilon_{(F' \sqcup F)_{\mathrm{left}}} \otimes (x \wedge y) \\
&= -(-1)^{|x|\pi(F) + \pi(F') + |x|} a \bullet_1 \varepsilon_{(F' \sqcup F)_{\mathrm{left}}} \otimes (x \wedge y)
\end{align*}
if $a$ is an $E$ generator (the extra minus sign arises for the same reason that it does in the proof of \cite[Theorem 1.2]{ManionDHASigns}). Similarly, $\Phi$ sends 
\[
(\varepsilon_{F'} \otimes x) \otimes a(\varepsilon_{F} \otimes y) = (-1)^{\pi(F)} (\varepsilon_{F'} \otimes x) \otimes (\varepsilon_{F} \otimes ay)
\]
to
\[
(-1)^{\pi(F) + |x|\pi(F)} \varepsilon_{(F' \sqcup F)_{\mathrm{left}}} \otimes (x \wedge ay)
\]
where $ay$ is still computed in $\Zb^P_{\delta,\pi}(F)$. In terms of the left action $\bullet_2$ (coming from $F$) of $A(M_2)$ on $\Zb^P_{\delta,\pi}((F' \sqcup F)_{\mathrm{left}})$, we can write this element as either
\begin{align*}
&-(-1)^{\pi(F) + |x|\pi(F) + \pi(F') + \pi(F) + |x|} a \bullet_2\varepsilon_{(F' \sqcup F)_{\mathrm{left}}} \otimes (x \wedge y) \\
&= -(-1)^{|x|\pi(F) + \pi(F') + |x|} a \bullet_2\varepsilon_{(F' \sqcup F)_{\mathrm{left}}} \otimes (x \wedge y)
\end{align*}
if $a$ is an $F$ generator (the extra minus sign arises because the orientation of the circle added by $a$ in the $F$ case is the opposite of the orientation of the circle added by $a$ in the $(F' \sqcup F)_{\mathrm{left}}$ case) or
\begin{align*}
&(-1)^{\pi(F) + |x|\pi(F) + \pi(F') + \pi(F) + |x|} a \bullet_2\varepsilon_{(F' \sqcup F)_{\mathrm{left}}} \otimes (x \wedge y) \\
&= (-1)^{|x|\pi(F) + \pi(F') + |x|} a \bullet_2\varepsilon_{(F' \sqcup F)_{\mathrm{left}}} \otimes (x \wedge y)
\end{align*}
if $a$ is an $E$ generator.

Thus, we get an isomorphism between \eqref{eq:CompositionTensor} and
\begin{equation}\label{eq:LeftSide}
\mathrm{Right}_{A_{EF}(M_1)} \left[ \frac{\Zb^P_{\delta,\pi}((F' \sqcup F)_{\mathrm{left}}))}{\spann\{a \bullet_1 \left(\varepsilon_{(F \sqcup F')_{\mathrm{left}}} \otimes z \right) + a \bullet_2 \left(\varepsilon_{(F \sqcup F')_{\mathrm{left}}} \otimes z\right)\}} \right]
\end{equation}
where, in the denominator, $a$ is an arbitrary multiplicative generator of $A_{EF}(M_2)$ and $z$ is an arbitrary basis element of $\Zb^P_{\delta,\pi}((F' \sqcup F)_{\mathrm{left}}))$. The isomorphism is compatible with the left actions of $A_{EF}(M_3)$; it is also compatible with the left actions of $F$ for incoming circles of $F$, and it relates the left actions of $E$ for incoming intervals or circles of $F$ by a minus sign.

The denominator is a sum of subspaces:
\begin{itemize}
    \item $\im(E_1 + E_2)$ for each interval component of $M_2$ with associated $E$-endomorphisms $E_1, E_2$ of $\Zb^P_{\delta,\pi}((F' \sqcup F)_{\mathrm{left}}))$, and
    \item $\im(E_1 + E_2) + \im(F_1 + F_2)$ for each circle component of $M_2$ with associated $E$-endomorphisms $E_1,E_2$ and $F$-endomorphisms $F_1,F_2$ of $\Zb^P_{\delta,\pi}((F' \sqcup F)_{\mathrm{left}}))$.
\end{itemize}
Taking the quotients one component of $M_2$ at a time and applying Lemma~\ref{lem:IntervalGluing} for interval components and Lemma~\ref{lem:CircleGluing} for circle components, the quotient in \eqref{eq:LeftSide} is isomorphic to
\[
\Zb^P_{\delta,\pi}((F' \circ F)_{\mathrm{left}})
\]
since $(F' \circ F)_{\mathrm{left}}$ is the surface obtained by doing all these gluings to $(F' \sqcup F)_{\mathrm{left}}$. It follows that the right side of the isomorphism in the statement of the corollary is isomorphic to 
\[
\mathrm{Right}_{A_{EF}(M_1)} \left[ \Zb^P_{\delta,\pi}((F' \circ F)_{\mathrm{left}}) \right],
\]
compatibly with the left actions of $A_{EF}(M_3)$ and the right actions of $F$ generators of $A_{EF}(M_1)$ and relating the right actions of $E$ generators of $A_{EF}(M_1)$ by a minus sign. Equivalently, it is isomorphic to
\[
\Zb^P_{\delta,\pi}(F' \circ F)
\]
compatibly with the left action of $A_{EF}(M_3)$ and the right action of $A_{EF}(M_1)$ as desired.
\end{proof}

\section{Degree and parity shifts}\label{sec:GradingsParities}

We explain here why we chose the formula \eqref{eq:DeltaADef} for $\delta_A$. As in \cite[Section 4]{ManionDHASigns}, we will postulate a general formula
\[
\delta = C_1 k_1 + \cdots + C_9 k_9
\]
for the grading shift associated to a sutured surface $F$, where
\begin{itemize}
    \item $k_1$ is the number of components
    \item $k_2$ is the genus (sum over all components)
    \item $k_3$ is the number of closed components
    \item $k_4$ is the number of non-closed components without $S_+$
    \item $k_5$ is the number of non-closed components without $S_-$
    \item $k_6$ is the number of $S_+$ intervals
    \item $k_7$ is the number of $S_+$ circles 
    \item $k_8$ is the number of $S_-$ circles
    \item $k_9$ is the number of boundary circles of $F$ with both $S_+$ and $S_-$.
\end{itemize}
As in \cite{ManionDHASigns}, $\delta$ is compatible with Lemma~\ref{lem:IntervalGluing} (interval gluing) if and only if the following system of equations is satisfied:
\begin{align}
    -C_1 + C_4 - 2C_6 + C_8 - 2C_9 &= 0 \quad \textrm{ (Case 1-1)} \label{it:Req1} \\ 
    -C_1 -2C_6 -C_9 &= 1 \quad \textrm{ (Case 1-2 or 1-3, no $S_-$ circle created)} \label{it:Req2} \\
    -C_1 -2C_6 + C_8 - 2C_9 &= 1 \quad \textrm{ (Cases 1-2 or 1-3, one $S_-$ circle created)}\label{it:Req2.5} \\
    -2C_6 + C_9 &= 1 \quad \textrm{ (Case 2-1a, no $S_-$ circle created)} \label{it:Req3} \\
    -2C_6 + C_8 &= 1 \quad \textrm{ (Case 2-1a, one $S_-$ circle created)} \label{it:Req4} \\
    -2C_6 + 2C_8 - C_9 &= 1 \quad \textrm{ (Case 2-1a, two $S_-$ circles created)} \label{it:Req5} \\
    C_4 - 2C_6 + 2C_8 - C_9 &= 0 \quad \textrm{ (Case 2-1b)} \label{it:Req6} \\
    C_2 - 2C_6 - C_9 &= 1 \quad \textrm{ (Case 2-2a, no $S_-$ circle created)} \label{it:Req7} \\
    C_2 - 2C_6 + C_8 - 2C_9 &= 1 \quad \textrm{ (Case 2-2a, one $S_-$ circle created)} \label{it:Req8} \\
    C_2 + C_4 - 2C_6 + C_8 - 2C_9 &= 0 \quad \textrm{ (Case 2-2b)} \label{it:Req9}
\end{align}
The general family of solutions is given by
\begin{equation}\label{eq:IntervalGeneralFamily}
(-2C_9, \, 2C_9, \, C_3, \, -1, \, C_5, \, (C_9 - 1)/2), \, C_7, \, C_9, \, C_9)
\end{equation}
for arbitrary values of $C_3,C_5,C_7,C_9$. Now, by examining the proof of Lemma~\ref{lem:CircleGluing}, we see that $\delta$ is compatible with circle gluing if and only if the following system of equations is satisfied:
\begin{align}
    -C_1 + C_3 - 2C_5 - 2C_7 &= 0 \quad \textrm{ (Case 1-1a)} \label{it:CircleReq1}\\
    -C_1 + C_4 - C_5 - 2C_7 &= 0 \quad \textrm{ (Case 1-1b)} \label{it:CircleReq2}\\
    -C_1 + C_4 - 2C_7 &= 0  \quad \textrm{ (Case 1-1c)} \label{it:CircleReq3}\\
    -C_1 - C_5 - 2C_7 &= 1 \quad \textrm{ (Case 1-2a)} \label{it:CircleReq4}\\
    -C_1  - C_5 - 2C_7 &=1 \quad \textrm{ (Case 1-2b or 1-3, no $S_-$ near $C_1$)} \label{it:CircleReq5}\\
    -C_1 - 2C_7 &= 1 \quad \textrm{ (Case 1-2b or 1-3, $S_-$ near $C_1$)} \label{it:CircleReq6}\\
    C_2 + C_3 - C_5 - 2C_7 &= 0 \quad \textrm{ (Case 2-1)} \label{it:CircleReq7}\\
    C_2 + C_4 - 2C_7 &= 0 \quad \textrm{ (Case 2-2)} \label{it:CircleReq8}\\
    C_2 - 2C_7 &= 1 \quad \textrm{ (Case 2-3)}. \label{it:CircleReq9}
\end{align}
From \eqref{it:CircleReq2} and \eqref{it:CircleReq3} we get $C_5 = 0$. Substituting \eqref{eq:IntervalGeneralFamily} into \eqref{it:CircleReq1}--\eqref{it:CircleReq9}, we get:
\begin{align*}
    2C_9 + C_3 - 2C_7 &= 0 \\
    2C_9 - 1 - 2C_7 &= 0 \\
    2C_9 -1 -2C_7 &= 0 \\
    2C_9 - 2C_7 &= 1 \\
    2C_9 - 2C_7 &= 1 \\
    2C_9 - 2C_7 &= 1 \\
    2C_9 + C_3 - 2C_7 &= 0 \\
    2C_9 - 1 - 2C_7 &= 0 \\
    2C_9 - 2C_7 &= 1.
\end{align*}
These equations hold if and only if $C_7 = C_9 - 1/2$ and $C_3 = -1$. Letting $C = C_9$, we arrive at the general family
\[
(-2C, \, 2C, \, -1, \, -1, \, 0, \, (C-1)/2, \, C - 1/2, \, C, \, C)
\]
of solutions to all the above equations together, so we can take
\[
\delta = -2C k_1 + 2Ck_2 - k_3 - k_4 + ((C-1)/2)k_6 + (C - 1/2)k_7 + C k_8 + C k_9.
\]
We also have
\[
h = -2k_1 + 2k_2 + 2k_3 + k_4 + k_5 + k_6 + k_7 + k_8 + k_9,
\]
so
\[
Ch = -2Ck_1 + 2Ck_2 + 2Ck_3 + Ck_4 + Ck_5 + Ck_6 + Ck_7 + Ck_8 + Ck_9.
\]
We can thus rewrite $\delta$ as
\[
\delta = Ch + (-2C - 1)k_3   + (-C-1)k_4 + (-C) k_5 + ((-C-1)/2)k_6 - (1/2) k_7.
\]
Letting $A = -C$, we get 
\begin{equation}\label{eq:DeltaADefAgain}
\delta = -Ah + (2A - 1)k_3 + (A-1)k_4 + Ak_5 + ((A-1)/2)k_6 - (1/2)k_7
\end{equation}
which recovers equation \eqref{eq:DeltaADef} for $\delta_A$.

\begin{proposition}\label{prop:NoAWorksForAllF}
    There is no $A \in \C$ such that \eqref{eq:DeltaADefAgain} is an integer for all sutured surfaces.
\end{proposition}

\begin{proof}
    Evaluating \eqref{eq:DeltaADefAgain} on a disk with boundary in $S_-$, we get $\delta = A-1$, so $A$ must be an integer. Evaluating on a disk with boundary in $S^+$ instead, we get $\delta = A - 1/2$, so $A$ cannot be an integer.
\end{proof}

\bibliographystyle{alpha}
\bibliography{biblio.bib}

\end{document}